\documentclass[12pt]{amsart}

\usepackage{epsfig}
\usepackage{amsmath}
\usepackage{amssymb,amsthm}
\usepackage{graphicx}

\usepackage{pst-node}

\usepackage{tikz}
\usepackage{enumerate}

\usepackage{hyperref}
\usepackage{a4wide}
\usepackage[normalem]{ulem}

\newtheorem{propo}{Proposition}[section]

\newtheorem{lemma}[propo]{Lemma}

\newtheorem{theo}[propo]{Theorem}

\newtheorem{prop}[propo]{Proposition}

\newcommand{\bl}{\begin{lemma}}
\newcommand{\el}{\end{lemma}}

\def\d12{{_{12}}}

\def\Cay{{\rm Cay}}

\def\Cos{{\rm Cos}}

\usepackage{indentfirst,latexsym,bm}

\def\Aut{{\rm Aut}}

\def\diam{{\rm diam}}

\def\K{{\rm K}}
\def\AGL{{\rm AGL}}

\def\H{{\rm H}}

\newcommand\PG{{\rm PG}}

\begin{document}
\title{Two-arc-transitive bicirculants
}

\thanks{Supported by the NNSF of China (12061034,12071484) and  NSF of Jiangxi (20192ACBL21007, GJJ190273)}

\thanks{The author is  grateful to Professor Shao Fei Du  for the  discussion
and comments on this paper.}

\author[W. Jin]{Wei Jin}
 \address{Wei Jin\\School of Mathematics and Statistics\\
Central South University\\
Changsha, Hunan, 410075, P.R.China}
\address{School of Statistics\\
Jiangxi University of Finance and Economics\\
 Nanchang, Jiangxi, 330013, P.R.China}
\email{jinweipei82@163.com}


\maketitle

\begin{abstract}
 In this paper, we  determine the class of finite 2-arc-transitive bicirculants.
We show that  a connected $2$-arc-transitive  bicirculant  is   one of the following graphs:
$C_{2n}$ where $n\geqslant 2$, $\K_{2n}$ where $n\geqslant 2$, $\K_{n,n}$ where $n\geqslant 3$, $ \K_{n,n}-n\K_2$ where $n\geqslant 4$,
 $B(\PG(d-1,q))$ and $B'(\PG(d-1,q))$ where $d\geq 3$ and $q$ is a prime power,
 $ X_1(4,q)$ where $q\equiv 3\pmod{4}$ is a prime power,
    $\K_{q+1}^{2d}$ where   $q$ is an odd prime power and $d\geq 2$ dividing   $q-1$,
$ AT_Q(1+q,2d)$ where $d\mid q-1$ and $d\nmid \frac{1}{2}(q-1)$,
$ AT_D(1+q,2d)$ where  $d\mid \frac{1}{2}(q-1)$ and $d\geq 2$,
    $\Gamma(d, q, r)$, where $d\geq 2$,  $q$ is a prime power and $r|q-1$, Petersen graph, Desargues graph,  dodecahedron graph,
folded $5$-cube, $X(3,2)$,  $ X_2(3)$, $ AT_Q(4,12)$,  $GP(12,5)$, $GP(24,5)$, $B(H(11))$,  $B'(H(11))$,
 $ AT_D(4,6)$ and $ AT_D(5,6)$.

\end{abstract}

\vspace{2mm}

\bigskip
\bigskip

\section{Introduction}

For a positive integer $s$, an $s$-arc of a graph $\Gamma$ is a sequence of
vertices $(v_0,v_1,\ldots,v_s)$  in  $\Gamma$ such that
$v_i,v_{i+1}$ are adjacent and $v_{j-1}\neq v_{j+1}$ where $0\leq
i\leq s-1$ and $1\leq j\leq s-1$. In particular, 1-arcs are called
\emph{arcs}.
The graph $\Gamma$ is said to be \emph{$(G,s)$-arc-transitive}
if, for each $i\leq s$, its graph automorphism subgroup $G$
 is transitive on the set of  $i$-arcs.
Moreover,  if $G$ is the whole graph automorphism group, then
$G$ is often omitted and we write simply \emph{$s$-arc-transitive}.
The first remarkable result about $s$-arc-transitive graphs comes
from Tutte  \cite{Tutte-1,Tutte-2}, and since then, this family of graphs has
been studied extensively, see
\cite{IP-1,Li-abeliancay-2008,weiss}.

A graph with $2n$ vertices   is said to be  a \emph{bicirculant} if its automorphism group contains  an element  $g$ of order $n$ which fixes no vertex and has  exactly two length $n$ cycles
 on the vertex set. The family  of bicirculants includes Cayley graphs of dihedral groups, the generalized Petersen graphs \cite{FGW-1971}, Rose-Window graphs \cite{Wilson} and Taba\v{c}jn graphs \cite{AHKOS-2015} (particularly, for the last  three families of graphs,  the two orbits of the cyclic group $\langle g\rangle$ form two  cycles in the graph).

Some classification results on
arc-transitive bicirculants with small valency have been achieved. For instance, combining together
the results of \cite{FGW-1971,MP-2000,P-2007} one obtains the classification of all cubic arc-transitive bicirculants.  Similarly, the
classification of tetravalent arc-transitive bicirculants is achieved by  the results of \cite{KKM-2010,KKM,KKMW-2012}.
Recently, the determination of pentavalent arc-transitive bicirculants was also obtained \cite{AHK-2015,AHKOS-2015}.
Jajcay and his colleagues in \cite{JMSV} investigated  higher valency bicirculants,
 the automorphisms of bicirculants on $2p$ vertices for $p$ a prime are well understood in \cite{MMSF-2007}, while  the family of finite arc-transitive bicirculants is characterized in \cite{DGJ-2019}.

We call a graph with $n$ vertices  a {\it circulant} if it has an automorphism $g$ that is an  $n$-cycle.
Alspach and his colleagues determined the class
of finite 2-arc-transitive circulants in  \cite[Theorem 1.1]{ACMX-1996}.
In this paper, we extend this classification result to the family of 2-arc-transitive bicirculants, that is,
we will  find out all   the  2-arc-transitive bicirculants, and this
result is also a generalization of the classification result of
\cite{DMM-2008}.


\medskip
Our main theorem  precisely determines the family of  2-arc-transitive bicirculants.

\begin{theo}\label{th-quot-miniblock-1}
Let $\Gamma$ be a connected $2$-arc-transitive  bicirculant.
Then    $\Gamma$ is one of the following graphs:

 \begin{itemize}

\item[(1)]  $C_{2n}$ where $n\geqslant 2$;

\item[(2)]  $\K_{2n}$ where $n\geqslant 2$;

\item[(3)]   $\K_{n,n}$ where $n\geqslant 3$;

\item[(4)] $ \K_{n,n}-n\K_2$ where $n\geqslant 4$;

\item[(5)] $B(\PG(d-1,q))$ and $B'(\PG(d-1,q))$, where $d\geq 3$, $q$ is a prime power;

\item[(6)] $ X_1(4,q)$ where $q\equiv 3\pmod{4}$ is a prime power;

\item[(7)]    $\K_{q+1}^{2d}$ for some $d\geq 2$ dividing   $q-1$, where $q$ is an odd prime power;

\item[(8)]  $ AT_Q(1+q,2d)$, where $d\mid q-1$ and $d\nmid \frac{1}{2}(q-1)$;

\item[(9)]  $ AT_D(1+q,2d)$, where  $d\mid \frac{1}{2}(q-1)$ and $d\geq 2$;

\item[(10)] $\Gamma(d, q, r)$, where $d\geq 2$,  $q$ is a prime power and $r|q-1$;

\item[(11)]  Petersen graph;

\item[(12)]  Desargues graph;

\item[(13)]  dodecahedron graph;

\item[(14)]  folded $5$-cube;

\item[(15)] $X(3, 2)$;

\item[(16)] $ X_2(3)$;

\item[(17)]  $ AT_Q(4,12)$;

\item[(18)]  $GP(12,5)$ and $GP(24,5)$;

\item[(19)]  $B(H(11))$ and $B'(H(11))$;

\item[(20)]    $ AT_D(4,6)$ and $ AT_D(5,6)$.

\end{itemize}

\end{theo}

\bigskip

The definitions of the graphs arising in Theorem  \ref{th-quot-miniblock-1} will be given in the next section.
The  graphs from (1)-(10) are   infinite families of 2-arc-transitive graphs, and graphs from (11)-(20) are sporadic 2-arc-transitive graphs.


This paper is organized  as follows.
After this introduction, we give, in Section 2,
some definitions on groups and graphs that we need and also prove some elementary  lemmas which will be used in the following analysis.
Let $\Gamma$ be a connected $2$-arc-transitive  bicirculant  with at least three vertices.
Let $N$ be a  normal subgroup of $G$ maximal with respect to having  at least $3$ orbits.
It is proved in Theorem \ref{bic-reduction-th2} that:
$\Gamma$ is a cyclic or metacyclic cover of $\Gamma_N$,
$\Gamma_N$ is $(G/N,2)$-arc-transitive and  is either a circulant or a bicirculant, and $G/N$ is faithful and either quasiprimitive or bi-quasiprimitive on $V(\Gamma_N)$. Moreover,  all the base graphs are determined.
The family of cyclic covers of base graphs is determined in Section 4, and
the family of metacyclic covers of base graphs is determined in Section 5.

\bigskip


\section{Preliminaries}

In this section, we give some definitions about groups and graphs and also prove some  results which will be used in the
following discussion.


We denote by $\mathbb{Z}_n$ the cyclic group of order $n$.
Let $G$ be a permutation group on a set $\Omega$ and $\alpha \in \Omega$. Denote by $G_\alpha$ the stabilizer of  $\alpha$ in $G$, that is, the subgroup of $G$ fixing the point $\alpha$. We say that $G$ is \emph{semiregular} on
$\Omega$ if $G_\alpha=1$ for every $\alpha\in \Omega$ and \emph{regular} if $G$ is transitive and semiregular.

Let  $G$ be a transitive permutation group on a set  $\Omega$ and
let $\mathcal{B}$ be a $G$-invariant partition of $\Omega$. If the only possibilities for $\mathcal{B}$ are the partition into one part, or the partition into singletons then $G$ is called \emph{primitive}. The \emph{kernel} of $G$ on $\mathcal{B}$ is the normal subgroup of $G$ consisting of all elements that fix setwise each $B\in\mathcal{B}$.
Let $B$ be a non-empty
subset of $\Omega$. Then $B$ is called a \emph{block} of
$G$ if, for any $g\in G$, either $B^g=B$ or
$B^g\cap B =\varnothing$.
If $N$ is an intransitive normal subgroup of $G$, then each $N$-orbit is a block of $G$, and the
set of $N$-orbits $\mathcal{B}$   forms a
$G$-invariant partition of $\Omega$.

\subsection{Graph theoretic notions}

All graphs in this paper are finite, simple, connected and undirected. For a graph $\Gamma$, we use  $V(\Gamma)$ and
$\Aut(\Gamma)$ to denote its \emph{vertex set}  and
\emph{automorphism group}, respectively.  For the group theoretic terminology not defined here we refer the reader to \cite{Cameron-1,DM-1,Wielandt-book}.


 For a graph $\Gamma$, its
\emph{complement} $\overline{\Gamma}$ is the graph with vertex set $V(\Gamma)$,
and two vertices are adjacent if and only if they are not adjacent in $\Gamma$.
We denote the complete graph on $n$ vertices by $\K_n$.

Let $\Gamma$ be a graph, and $G\leqslant \Aut (\Gamma)$. Suppose
$\mathcal{B}=\{B_1,B_2,\ldots,B_n \}$ is a partition of $V(\Gamma)$.
The \emph{quotient graph} $\Gamma_{\mathcal{B}}$ of $\Gamma$
relative to $\mathcal{B}$ is defined to be the graph with vertex set
$\mathcal{B}$ such that $\{B_i,B_j\}$ is an edge of
$\Gamma_{\mathcal{B}}$ if and only if there exist $x\in B_i, y\in
B_j$ such that $\{x,y\} \in E(\Gamma)$. We say that
$\Gamma_{\mathcal{B}}$ is \emph{non-trivial} if $1< |\mathcal{B}|<
|V(\Gamma)|$. If for any two adjacent blocks $B_i,B_j$ and $x\in
B_i$, we have $|\Gamma(x)\cap B_j|=1$, then we say that $\Gamma$ is
a \emph{cover} of $\Gamma_{\mathcal{B}}$. If $\mathcal{B}$ is
$G$-invariant, then $G$ has an induced action on
$\Gamma_{\mathcal{B}}$ as a subgroup of automorphisms. Whenever the
blocks in $\mathcal{B}$ are the $N$-orbits, for some non-trivial
normal subgroup  $N$ of $G$, we will write
$\Gamma_{\mathcal{B}}=\Gamma_{N}$, and call it a \emph{$G$-normal
quotient} of $\Gamma$.   Moreover, if $\Gamma$ is a cover of
$\Gamma_{\mathcal{B}}$ and $\Gamma_{\mathcal{B}}=\Gamma_{N}$, then
$\Gamma$  is called a \emph{$G$-normal cover} of
$\Gamma_{\mathcal{B}}$. Suppose that  $\Gamma$ is  a  cover of $\Gamma_{\mathcal{B}}$. If the kernel of the $G$-action on $\mathcal{B}$ is a cyclic or metacyclic group, then
$\Gamma$ is called a \emph{cyclic cover}  or \emph{metacyclic cover} of $\Gamma_{\mathcal{B}}$, respectively.

If $G\leq \Aut(\Gamma)$ is vertex-transitive on a bipartite graph $\Gamma$, then $G$ has a subgroup $G^+$ of index 2 with two orbits on $V(\Gamma)$, say $\Delta_1$ and $\Delta_2$, the two parts of the bipartition of $V(\Gamma)$. We shall use this notation, for $G^+$, $\Delta_1$ and $\Delta_2$, throughout the paper.
A group acting on a set $\Delta$ is said to be \emph{faithful} on $\Delta$ if the only element which fixes $\Delta$
pointwise is the identity.

For a finite group $T$, and a subset $S$ of $T$ such that $1\notin
S$ and $S=S^{-1}$, the \emph{Cayley graph} $\Cay(T,S)$ of $T$ with
respect to $S$ is  the graph with vertex set $T$ and edge set
$\{\{g,sg\} \,|\,g\in T,s\in S\}$. In particular, $\Cay(T,S)$ is
connected if and only if $T=\langle S\rangle$.
The group $R(T) = \{
\sigma_t|t\in T\}$ of right multiplications $\sigma_t : x \mapsto xt$
is a subgroup of the automorphism group $\Aut(\Gamma)$ and acts
regularly on the vertex set. Indeed, a graph is a Cayley graph if and only if it admits a regular group of automorphisms. We note that circulants are precisely the Cayley graphs for cyclic groups.

A graph $\Gamma$ is said to be a \emph{bi-Cayley graph} over a group $H$ if it admits $H$ as a semiregular automorphism
group with two orbits of equal size. (Some authors have used the term \emph{semi-Cayley} instead, see \cite{LM-1993,RJ-1992}.) Moreover, bicirculants are exactly the bi-Cayley graphs over cyclic groups.
The family of  bi-Cayley graphs has been extensively studied, for example cubic bi-Cayley graphs over abelian groups were investigated by Zhou and Feng \cite{ZF-bicay-2014} while the automorphism groups of bi-Cayley graphs were studied in \cite{ZF-bicir-2016}.

Let $G$ be a group and let $L,R$ be two subgroups of $G$. Let $D$ be a union of double cosets of $L$ and $R$ in $G$, that is, $D=\bigcup_iRd_iL$. Denote by $[G:L]$ and $[G:R]$ the set of cosets of $L$ and $R$ in $G$, respectively. Define a bipartite graph $\Cos(G,L,R,D)$ with two bipartite halves $[G:L]$ and $[G:R]$, and edge set
$\{\{Lg,Rdg\}|g\in G,d\in D\}$. This graph is called the \emph{bi-coset graph} of $G$ with respect to $L,R$ and $D$.

Let $\Gamma$ be a $G$-vertex-transitive  bicirculant over a cyclic subgroup $H$ of $G$. Then the first lemma presents a
relationship between the two orbits of  $H$  and the blocks for the action of $G$ on the vertex set.

\begin{lemma}{\rm(\cite[Lemma 3.1]{DGJ-2019})}\label{quo-1}
Let $\Gamma$ be a $G$-vertex-transitive   bi-Cayley graph over the  subgroup $H$ of $G$.
Let $H_0$ and $H_1$ be the two orbits of $H$ on $V(\Gamma)$ and
let $\mathcal{B}$ be a $G$-invariant partition of $V(\Gamma)$. Then the following hold.

\begin{enumerate}[$(1)$]
\item Either all elements of $\mathcal{B}$ are subsets of $H_0$ or $H_1$;
or  $B\cap H_0\neq \varnothing$ and $B\cap H_1\neq \varnothing$ for every   $B\in \mathcal{B}$.

\item If $B\in \mathcal{B}$ and  $B\cap H_i\neq \varnothing$ for some $i$, then $B\cap H_i$ is a block for $H$ on $H_i$.
\end{enumerate}

\end{lemma}

The \emph{generalized Petersen graph} $GP(n,r)$  is the graph on the vertex set
$$\{u_0,u_1,\ldots,u_{n-1},v_0,v_1,\ldots,v_{n-1}\}$$
with the adjacencies:
$$u_i\sim u_{i+1}, \, v_i \sim v_{i+r}, \, u_i\sim v_i, \quad i=0,1,\ldots,n-1.$$
Hence each  generalized Petersen graph $GP(n,r)$ has valency 3.
By \cite{FGW-1971},
$GP(n,r)$ is arc-transitive if and only if   $(n,r)=(4,1)$, $(5, 2) $, $(8, 3)$, $ (10, 2),$ $ (10, 3), (12, 5) $  and $  (24, 5)$.


Gross and Tucker \cite{GT-1987} introduced a  combinatorial method  of  cover graphs through a voltage graph.
A \emph{graph homomorphism} $f$ from a graph $\widetilde{X}$ to a graph $X$ is a mapping from the vertex
set $V ( \widetilde{X} )$ to the vertex set $V (X)$ such that if $\{u, v\}\in  E( \widetilde{X} )$ then $\{f(u), f(v)\}\in  E(X)$.
When $f$ is surjective, $X$ is called a \emph{quotient of $\widetilde{X}$}. For $v\in  V (X)$, let $X(v)$ denote the
set of neighbours of $v$ in $X$. A \emph{covering projection} is defined as a graph homomorphism
$p : \widetilde{X} \mapsto X$ which is surjective and locally bijective, which means that the restriction
$p : \widetilde{X}(\widetilde{v}) \mapsto X(v)$ is a bijection, whenever $\widetilde{v}$ is a vertex of $\widetilde{X}$ such that $p(\widetilde{v}) = v \in V (X)$.
We call $X$ the \emph{base graph}, $\widetilde{X}$ a \emph{cover graph} of $X$, and the pre-images
$p^{-1}(v), v\in  V (X)$ the \emph{fibres}. A covering projection $p : \widetilde{X} \mapsto X$ is called \emph{regular} if there
exists a semi-regular subgroup $N$ of the group of automorphisms $\Aut( \widetilde{X} )$ of $\widetilde{X}$ such that
the quotient graph $\widetilde{X}_N$ (with vertices taken as the orbits of $N$ on $V( \widetilde{X} ))$ is isomorphic
to $X$. In that case we call $N$ the \emph{covering (transformation) group}.

The above properties can be exploited to construct regular cover graphs of a given
graph, as follows.

Let $X$ be a connected graph, and let $N$ be a finite group. Suppose $\psi : Arc(X) \mapsto N$ is
a function assigning a group element to each arc of $X$, such that $\psi(v, u) = (\psi(u, v))^{-1}$ for
every arc $(u, v)\in  Arc(X)$. Here $\psi$ is called a \emph{voltage assignment},  the values of $\psi$ are called \emph{voltages}, and $N$ is  the covering
group. In particular, $\psi$ is
called \emph{reduced} if the values of $\psi$ on a spanning tree are trivial (equal to the identity element
of $N$). We may construct a larger graph $Cov(X,\psi)=X\times_{\psi} N$, called the \emph{(derived) voltage graph}
(or \emph{cover graph}), with vertex set $V (X) \times N$ and adjacency defined by $(u, g) \sim  (v, h)$
if and only if $u \sim v$ and $h = g\psi (u, v)$.

\begin{lemma}{\rm (\cite[Lemma 2.6]{DMM-2008}) }\label{2at-dih-lem1}
Let $X\mapsto Y$ be a regular cyclic covering of a connected graph such that some $2$-arc-transitive group $G\leq \Aut(X)$
projects along $X\mapsto Y$. Then there exists a regular prime cyclic covering
$X'\mapsto Y$ such that some $2$-arc-transitive group $G'\leq \Aut(X')$ projects along
$X'\mapsto Y$.

\end{lemma}

The following fact is  well known, and can be easily proved.





\begin{lemma}{\rm (\cite[Proposition 2.2]{DKX-2005},\cite[Proposition 2.3]{DMM-2008}) }\label{lem-cent-1}
Let $K$ be a finite group, let $X = Y \times_f K$ be a connected regular cover of
a graph $Y$ derived from a voltage assignment $f$ with the voltage group $K$, and let the lifts
of $\alpha \in  \Aut(Y)$ centralize $K$, considered as the group of covering transformations. Then for any
closed walk $W$ in $Y$, there exists $k\in  K$ such that $f(W^\alpha) = kf(W)k^{-1}$. In particular, if $K$ is abelian,
$f(W^\alpha) = f (W )$ for any closed walk $W$ of $Y$.

\end{lemma}

The \emph{cycle space} of a graph $\Gamma$ is the subspace of the vector space
generated by all simple cycles of $\Gamma$. The following lemma is obvious, see for instance \cite[Lemma 2.5]{DMM-2008}.

\begin{lemma}\label{lem-vanish}
Let $\Gamma$ be a graph and let $\mathcal{C}$ be a set of cycles of $\Gamma$ spanning the cycle space of $\Gamma$. If
$\overline{\Gamma}$ is a  cover graph of $\Gamma$ given by a voltage assignment $f$ for which each $C\in \mathcal{C}$ vanishes, then $\overline{\Gamma}$ is disconnected.

\end{lemma}

\subsection{The graphs appearing in Theorem \ref{th-quot-miniblock-1}}

For integers $d,r\geq 2$, the    \emph{Hamming graph}  $\H(d,r)$  has
vertex set $\Delta^d=\{(x_1,x_2,\ldots,x_d)|x_i\in \Delta\}$, where
$\Delta=\{0,1,\ldots,r-1\}$, and  two vertices $v$ and $v'$ are adjacent if and only if they are different
in exactly one coordinate.  The    Hamming graph  $\H(d,2)$ is  called a \emph{$d$-cube}, and a \emph{folded $d$-cube} is the quotient of $\H(d,2)$ by the partition into blocks of size two given by pairs of vertices at distance $d$.
Moreover, the  folded $5$-cube is also a strongly regular graph with parameters $(16,5,0,2)$), automorphism group $\mathbb{Z}_2^4:S_5$.
 The complement of the folded $5$-cube is the \emph{Clebsch graph}.

The \emph{dodecahedron} has both girth  and diameter 5,   automorphism group $A_5\times \mathbb{Z}_2$ and intersection array  $(3,2,1,1,1;1,1,1,2,3)$, see \cite[p.1]{BCN} and \cite{Frucht-1};
it is 2-arc-transitive.


Let $p$ be an odd prime and let $r$ be a positive   integer  dividing $p-1$. Let $A$ and $A'$ denote two disjoint copies of $\mathbb{Z}_p$ and denote the corresponding elements of $A$ and $A'$ by $i$ and $i'$, respectively.
Let $L(p,r)$ be the unique subgroup of the multiplicative group of $\mathbb{Z}_p$ of order $r$.   We define
two graphs,  $G(2p,r)$ and $G(2,p,r)$, with vertex set $A\cup A'$.  The graph $G(2p,r)$ has edge set $\{\{x,y'\}|x,y\in \mathbb{Z}_p,y-x\in L(p,r)\},$ while  the graph $G(2,p,r)$ (defined only for $r$ even) has  edge set $\{\{x,y\},\{x',y\},\{x,y'\},\{x',y'\}|x,y\in \mathbb{Z}_p,y-x\in L(p,r)\}.$
Note that $G(2p,r)$ is bipartite.

For each integer $d\geq 3$ and prime power $q$, let $B(\PG(d-1,q))$ be the bipartite graph with vertices the 1-dimensional and $(d-1)$-dimensional subspaces of a $d$-dimensional vector space over $GF(q)$, and two subspaces are adjacent if and only if one is contained in the other. We denote the bipartite complement of $B(\PG(d-1,q))$ by $B'(\PG(d-1,q))$, that is, the bipartite graph with the same vertex set but a 1-subspace and a $(d-1)$-subspace  are adjacent if and only if their intersection is the zero subspace.

We define $B'(H(11))$ to be the bipartite graph with vertices the elements of $\mathbb{Z}_{11}$ and the sets $R+i$, where $i\in\mathbb{Z}_{11}$ and
$R=\{1,3,4,5,9\}$, that is, the set of non-zero quadratic residues modulo 11,  such that $n\in\mathbb{Z}_{11}$  is adjacent to $R+i$ if and only if $n\notin R+i$. Denote the bipartite complement of $B'(H(11))$ by  $B(H(11))$.  We note that  $B(H(11))$ is isomorphic to $G(22,5)$, see for instance \cite[p.200]{CO-1987}.

Let $q$ be a prime power such that $q\equiv 3\pmod{4}$ and let
$F(q)=GF(q)$,  $S(q)$ be the set of all non-zero squares of $F(q)$
and $N(q)$ be the set of  non-zero non-squares of $F(q)$.
Let $Y=\K_{q+1}$, and   identify the vertices of $Y$ with the projective lines
$\PG(1, q)=F(q)\cup \{\infty\}$ (see \cite[p.283,285]{DMW-1998}).
 We define $X_1(4, q)$ to be the 4-fold
cover $Cov(Y, f )$, where the voltage $f: Arc(Y)\mapsto \mathbb{Z}_4$ is defined with the following
rule:

\begin{center}
$f(x,y): = \left\{
\begin{array}{lll}
0, & \mbox{if $\infty \in\{x,y\} $};\\
1, & \mbox{if $ y-x \in S(q)$};\\
3, & \mbox{if $y-x \in N(q)$}.\end{array} \right.$
\end{center}

If $q=3$, then   $X_1(4, q)$ is the generalized Petersen graph $GP(8,3)$, and is also known as the M\"obius-Kantor graph.

\begin{lemma}\label{lem-cov-comp-1}
Let $q$ be a prime power such that $q\equiv 3\pmod{4}$. Then the graph  $X_1(4, q)$  is a $2$-arc-transitive bicirculant.
\end{lemma}
\proof  By \cite[p.283,285,288]{DMW-1998} and \cite[Proposition 3.2]{DMW-1998}, we know that $X_1(4, q)$ is a $2$-arc-transitive Cayley graph of the group
$G=\langle r,t|r^4=1,t^{q+1}=r^2,t^{-1}rt=r^{-1}\rangle$. Thus  $t$ is an order $2(q+1)$ element of $G$ that  has exactly two orbits of size $2(q+1)$
on the set of vertices of $X_1(4, q)$.
Hence $X_1(4, q)$ is a bicirculant of $T=\langle t\rangle\cong \mathbb{Z}_{2(q+1)}$.
\qed

\medskip

Let $\Gamma=\K_{5,5}-5\K_2$, where
$V(\Gamma)=\{1,2,3,4,5\}\cup\{1',2',3',4',5'\}$,
and $E(\Gamma)=\{ij'|i\neq j,i,j'\in V(\Gamma)\}$.

Define $X_2(3)=\Gamma \times_f \mathbb{Z}_3$, with the voltage assignment
$f:Arc(\Gamma)\mapsto \mathbb{Z}_3$ such that
$$f_{1,2'}=f_{1,3'}=f_{1,4'}=f_{1,5'}=f_{2,1'}=f_{2,3'}=f_{3,1'}$$
$$=f_{3,2'}=f_{4,1'}=f_{4,5'}=f_{5,1'}=f_{5,4'}=0,$$
$$f_{2,5'}=f_{3,4'}=f_{4,3'}=f_{5,2'}=1,$$
$$f_{2,4'}=f_{3,5'}=f_{4,2'}=f_{5,3'}=2.$$

\begin{lemma}\label{lem-cov-k55}
The graph  $X_2(3)$   is a $2$-arc-transitive bicirculant.

\end{lemma}
\proof Let $\Gamma:=\K_{5,5}-5\K_2$. Then $G:=\Aut(\Gamma)=S_5\times S_2$.  Let
$G^+$ be the maximal subgroup of $\Aut(\Gamma)$ that fixes the two bipartite halves of $\Gamma$ setwisely. Then $|G:G^+|=2$.
By  \cite[Theorem 1.1]{XD-2014},  $G$ is transitive on the set of 2-arcs of $\Gamma$, and so $G^+$ is 2-transitive on both
the two bipartite halves. Thus we must have   $A_5\leq G^+\leq S_5$.
Let $W$ be the socle of $G^+$. Then $W\cong A_5$.

Let $N$ be the covering group and let   $\widetilde{W}$ be the lift of $W$. Then $N\cong \mathbb{Z}_3$ and $\widetilde{W}/N\cong W$, and so  $N\unlhd C_{\widetilde{W}}(N)$ as $N$ is cyclic.
It follows that  $C_{\widetilde{W}}(N)/N\lhd \widetilde{W}/N\cong A_5$, and hence $C_{\widetilde{W}}(N)/N$ is isomorphic to either  1 or  $A_5$.
Since $(\widetilde{W}/N)/(C_{\widetilde{W}}(N)/N)\cong \widetilde{W}/C_{\widetilde{W}}(N)$ is isomorphic to a subgroup of $\Aut(N)$
that is abelian, we know that
$C_{\widetilde{W}}(N)/N$ is not 1, and so
$C_{\widetilde{W}}(N)/N$ is isomorphic to  $W$. Thus
$C_{\widetilde{W}}(N)=\widetilde{W}$ and $N\leq Z(\widetilde{W})$.

Note that $A_5$  has an element $\alpha$ of order $5$ that is regular on
each bipartite half of $\Gamma$.
Let $M:=\langle \alpha \rangle\cong \mathbb{Z}_{5}$ and let  $\widetilde{M}$ be the lift of $M$. Then $M\leq W$, $\widetilde{M}\leq \widetilde{W}$ and $\widetilde{M}/N\cong M$.
Since   $(3,5)=1$,  we have $\widetilde{M}=N.M \cong \mathbb{Z}_{15}$, and so $\widetilde{M}$ is a cyclic group that has exactly two orbits of size 15 in the vertex set of $X_2(3)$, hence  $X_2(3)$  is a bicirculant over $\widetilde{M}$.
\qed

\medskip
Let $q = r^l$ for an odd prime $r$ and $GF(q)^*=\langle \theta \rangle$
the multiplicative group of the field $GF(q)$ of order $q$. Let $\K_{q+1,q+1}-(q + 1)\K_2$ be
a complete bipartite graph minus a matching whose vertex set is two copies of the
projective line $\PG(1, q)$, where the missing matching consists of all pairs $[i,i'], i\in \PG(1, q)$. For any $d | q-1$ and $d\geq 2$, define a voltage graph
$\K_{q+1}^{2d}= (\K_{q+1,q+1}-(q + 1)\K_2)\times_f \mathbb{Z}_d$, where

$$f_{\infty',i}=f_{\infty,j'}=\overline{0}, i,j\neq \infty;$$
$$f_{i,j'}=\overline{h}, j-i=\theta^h, i,j\neq \infty.$$

\begin{lemma}\label{lem-cov-comp-minus1}
Let $q\geq 7$ be an odd  prime power and let  $d|q-1$ and $d\geq 2$. Then the graph  $\K_{q+1}^{2d}$   is a $2$-arc-transitive bicirculant.


\end{lemma}
\proof Let $\overline{\Gamma}:=\K_{q+1}^{2d}$ where $d|q-1$ and $d\geq 2$.
If $q\equiv 3\pmod{4}$, or $q\equiv 1\pmod{4}$ and $d$ divides $\frac{q-1}{2}$, then by \cite[Theorem 1.2]{DMM-2008}, $\overline{\Gamma}$ is a
$2$-arc-transitive dihedrant,  and so $\overline{\Gamma}$ is a
$2$-arc-transitive bicirculant. From now on, we suppose that
$q\equiv 1\pmod{4}$, and $d$ divides $q-1$ but not $\frac{q-1}{2}$.

Let $\Gamma=\K_{q+1,q+1}-(q + 1)\K_2$ be the  complete bipartite graph minus a matching whose vertex set is two copies of the
projective line $\PG(1, q)$, where the missing matching consists of all pairs $[i,i'], i\in \PG(1, q)$.
 Let $$\sigma:i\mapsto i',i'\mapsto i.$$
Then $\sigma$ is an order 2 element of $\Aut(\Gamma)$.
By \cite[Lemma 2.9]{XD-2014}, $\PG L(2,q)\times \langle \sigma \rangle$ is a subgroup of
$\Aut(\Gamma)$ and can be lifted to a subgroup of $\Aut(\overline{\Gamma})$.
Let $\sigma'$ be the lift of $\sigma$.
Let $N=\langle k \rangle \cong \mathbb{Z}_{d}$ be the covering group.
Then $\langle \sigma' \rangle/N\cong \langle \sigma \rangle$.
Note that $N$ acts faithfully and permutation equivalently on each block. Consequently, $\langle \sigma' \rangle\cong \mathbb{Z}_{2d}$.

Let $\mathcal{B}=\{B_{1},B_{2},\ldots,B_{q+1};B_{1'},B_{2'},\ldots,B_{(q+1)'}\}$ be the block system of $V(\overline{\Gamma})$ such that $\overline{\Gamma}_{\mathcal{B}}\cong \Gamma$,
where $B_i$ corresponds to the vertex $i$ of $\Gamma$ and $B_i'$ corresponds to the vertex $i'$ of $\Gamma$.
Set $B_{i}=\{b_{i1},b_{i2},\ldots,b_{id}\}$ and $B_{i'}=\{b_{i'1},b_{i'2},\ldots,b_{i'd}\}$.
Then $(\sigma')^{d}$ is an order 2 element of $\Aut(\overline{\Gamma})$ such that $b_{ij}\mapsto b_{i'j}$ and $b_{i'j}\mapsto b_{ij}$.

Let $W=PSL(2,q)$ and let  $\widetilde{W}$ be the lift of $W$.
In light of the fact that  $q\geq 7$, $W$ is nonabelian and simple.
Since $C_{\widetilde{W}}(N)/N\lhd \widetilde{W}/N\cong PSL(2,q)$, it follows that $C_{\widetilde{W}}(N)/N$ is isomorphic to either  1 or  $PSL(2,q)$.
Since $(\widetilde{W}/N)/(C_{\widetilde{W}}(N)/N)\cong \widetilde{W}/C_{\widetilde{W}}(N)$ is isomorphic to a subgroup of $\Aut(N)$
that is abelian, we know that
$C_{\widetilde{W}}(N)/N$ is not 1, and so
$C_{\widetilde{W}}(N)/N$ is isomorphic to  $W=PSL(2,q)$. Thus
$C_{\widetilde{W}}(N)=\widetilde{W}$ and $N\leq Z(\widetilde{W})$.

Note that $PGL(2,q)$ is 3-transitive on $\PG(1,q)$ and also has an element $\alpha$ of order $q+1$ that is regular on
$\PG(1,q)$.
Let $M:=\langle \alpha^2 \rangle\cong \mathbb{Z}_{(q+1)/2}$ and let  $\widetilde{M}$ be the lift of $M$. Then $M\leq W$.
Since $q\equiv 1\pmod{4}$  is odd, it follows that  $(q-1,q+1)=2$, and so $(q-1,q+1/2)=1$, hence we get $(d,q+1/2)=1$.
Thus $\widetilde{M}=NM \cong \mathbb{Z}_{d(q+1)/2}$.
Moreover, by the definition of $(\sigma')^{d}$, we can see  that $\langle \widetilde{M},(\sigma')^{d} \rangle\cong \mathbb{Z}_{d(q+1)}$.
Therefore  $\overline{\Gamma}$ is a bicirculant over the cyclic group $\langle \widetilde{M},(\sigma')^{d} \rangle$.
\qed

\medskip
Alspach and his colleagues determined the family
of undirected 2-arc-transitive circulants in the following theorem.

\begin{theo}{\rm(\cite[Theorem 1.1]{ACMX-1996})}\label{circulant-2arctrans-1}
A connected, $2$-arc-transitive circulant of order $n$, $n\geq 3$,
is one of the following graphs:

 \begin{itemize}
\item[(1)]  the complete graph $\K_n$ which is exactly $2$-transitive;

\item[(2)]  the complete bipartite graph
$\K_{\frac{n}{2},\frac{n}{2}}$, $n\geq 6$, which is exactly
$3$-transitive;

\item[(3)] the complete bipartite graph $\K_{\frac{n}{2},\frac{n}{2}}$
minus a $1$-factor, $\frac{n}{2}\geq 5$ odd, which is exactly
$2$-transitive; and

\item[(4)]  the cycle $C_n$ of length $n$, which is $k$-arc transitive
for all $k\geq 0$.

\end{itemize}

\end{theo}

\medskip

The  result about primitive permutation groups  that contain an
element with exactly two equal cycles is due to M\"{u}ller \cite[Theorem 3.3]{Muller-2018}.
Using this  result, the following  proposition determines all the vertex primitive arc-transitive  bicirculants.

\begin{prop}{\rm(\cite[Proposition 4.2]{DGJ-2019})}\label{bicirculant-quasiprimitive-th1}
Let $\Gamma$ be a connected $G$-arc-transitive  bicirculant over  the cyclic subgroup $H$ of order $n$ such
that $G$ is quasiprimitive on $V(\Gamma)$. Then $\Gamma$ is one of the following graphs:

 \begin{itemize}
\item[(1)]  $\K_{2n}$ where $n\geq 2$;

\item[(2)]  Petersen graph or its complement;

\item[(3)] $\H(2,4)$ or its complement;
\item[(4)]  folded $5$-cube or its complement.

\end{itemize}

\end{prop}











\begin{lemma}{\rm(\cite[Lemma 5.3]{DGLP-locdt-2012})}\label{dt-quotient12}
Let $\Gamma$ be a connected locally $(G, s)$-distance transitive
graph with $s\geq 2$. Let $1\neq N \lhd G$ be intransitive on
$V(\Gamma)$, and let $\mathcal{B}$ be the set of $N$-orbits on
$V(\Gamma)$. Then one of the following holds:

{\rm (i)}  $|\mathcal{B}| = 2$.

{\rm(ii)} $\Gamma$ is bipartite, $\Gamma_N\cong \K_{1,r}$ with
$r\geq 2$ and $G$ is intransitive on $V(\Gamma)$.

{\rm(iii)} $s=2$, $\Gamma\cong \K_{m[b]}$,   $\Gamma_N \cong \K_{m}$
with $m\geq 3$ and $b\geq 2$.

{\rm(iv)}   $N$ is semiregular on $V(\Gamma)$,  $\Gamma$ is a cover
of $\Gamma_N$, $|V(\Gamma_N)|<|V(\Gamma)|$ and $\Gamma_N$ is
$(G/N,s')$-distance transitive where $s'=\min\{s,\diam(\Gamma_N)\}$.
\end{lemma}

Note that every $(G, 2)$-arc-transitive
graph is $(G, 2)$-distance-transitive. Thus the above lemma is also true for
$(G, 2)$-arc-transitive
graphs.

Define  $Q_{2n}=\langle u,v\rangle$, where
$u^n=1,v^2=u^{n/2},v^{-1}uv=u^{-1}$, the so-called \emph{generalized quaternion group} of order $2n$.

We define two covers of $\K_4$ with respective covering transformation groups
$N=\langle a,b\rangle \cong D_6$ and $Q_{12}$, where $V(\K_4)=\{1,2,3,4\}$.

(1) $AT_D(4,6)=\K_4\times_f D_6$, with the voltage assignment $f:Arc(\K_4)\mapsto D_6$ defined by
$$f_{1,2}=b,f_{1,3}=ba,f_{1,4}=ba^{-1},$$
$$f_{2,3}=ba^{-1},f_{2,4}=ba, f_{3,4}=b;$$

(2) $AT_Q(4,12)=\K_4\times_f Q_{12}$, with the voltage assignment $f:Arc(\K_4)\mapsto Q_{12}$ defined by
$$f_{1,2}=b,f_{1,3}=ba^2,f_{1,4}=ba^{4},$$
$$f_{2,3}=b,f_{2,4}=ba^3, f_{3,4}=b.$$

We define one cover of $\K_5$ with the covering transformation group
$N=\langle a,b\rangle \cong D_6$, where $V(\K_5)=\{1,2,3,4,5\}$.

(3) $AT_D(5,6)=\K_5\times_f D_6$, with the voltage assignment $f:Arc(\K_5)\mapsto D_6$ defined by
$$f_{1,2}=ab,f_{1,3}=b,f_{1,4}=ba,f_{1,5}=b,f_{2,3}=ba,$$
$$f_{2,4}=b, f_{2,5}=b,f_{3,4}=ab,f_{3,5}=b,f_{4,5}=b.$$

Let $GF(q)$ be the field of order $q$ where $q$ is an odd prime power, and let $GF(q)^*=\langle \theta \rangle$. We identify the vertex set of the complete graph $\K_{1+q}$ with the projective line $\PG(1, q) =GF(q) \cup \{\infty\}$. Then we define two families of arc-transitive covers of $\K_{1+q}$ with the respective covering transformation groups $N= \langle a,b\rangle=Q_{2d}$ and $D_{2d}$:

(4)  $ AT_Q(1+q,2d)=\K_{1+q}\times_f Q_{2d}$, where $d\mid q-1$ and $d\nmid \frac{1}{2}(q-1)$;

(5)  $ AT_D(1+q,2d)=\K_{1+1}\times_f D_{2d}$, where  $d\mid \frac{1}{2}(q-1)$ and $d\geq 2$.

For both covers, the voltage assignments $f:Arc(\K_{1+q})\mapsto N$ are given by:
$$f_{\infty,i}=b;f_{i,j}=ba^h \ \ if \ \  j-i=\theta^h \ for \  i,j\neq \infty.$$

\begin{theo}{\rm (\cite[Thoerem 1.1]{XDKX-2015})}\label{metacyclic-complete-th1}
Let $X$ be a connected regular cover of the complete graph $\K_n $ $(n \geq 4)$ whose covering transformation group $N$ is non-trivial metacyclic, and whose fibre-preserving automorphism group acts $2$-arc-transitively on $X$. Then $X$ is isomorphic to one of the following graphs:

\begin{itemize}
\item[(1)]   $\K_{n,n}-n\K_2$ with  $N\cong \mathbb{Z}_2$;

\item[(2)]  $n=4$, $ AT_D(4,6)$ with   $N\cong D_6$;

\item[(3)] $n=4$, $ AT_Q(4,12)$ with   $N\cong Q_{12}$;

\item[(4)] $n=5$, $ AT_D(5,6)$ with   $N\cong D_{6}$;

\item[(5)] $n=1+q\geq 4$, $ AT_Q(1+q,2d)$ with   $N\cong Q_{2d}$, where $d\mid q-1$ and $d\nmid \frac{1}{2}(q-1)$;

\item[(6)] $n=1+q\geq 6$, $ AT_D(1+q,2d)$ with   $N\cong D_{2d}$, where  $d\mid \frac{1}{2}(q-1)$ and $d\geq 2$.
\end{itemize}

\end{theo}

Let $d \geq  2$ be an integer and let $q$ be a prime power. Let $V$ be a vector space of dimension
$d$ over $GF(q)$, the finite field with $q$ elements. We define the set of non-zero vectors
in $V$ as the point set $\mathcal{P}$ and the set of affine hyperplanes in $V$ as the block set $\mathcal{B}$, i.e.,
$\mathcal{P}= \{x\in  V | x \neq 0\}$ and $ \mathcal{B}= \{x + H | H$ is a hyperplane in $V$ and  $x\in V \setminus  H\}$.

We make a partition $\mathcal{G}$ of $\mathcal{P}$ such that the collinear non-zero vectors in $V$ belong
to the same group in $\mathcal{G}$. Note that each group in $\mathcal{G}$ has size $q-1$. Then $(\mathcal{P}; \mathcal{G}; \mathcal{B})$ is
a $GDD(n,m; k; 0; \lambda_2)$, where $n = q-1, m = \frac{q^d-1}{q-1}$
 is the number of projective points
in $V$, $k = q^{d-1}$ is the number of affine hyperplanes containing a given non-zero vector,
and $\lambda_2 = q^{d-2}$ is the number of affine hyperplanes containing two given non-zero and
non-collinear vectors.

We look at the dual incidence structure $\mathcal{I}^*$ of $\mathcal{I} = (\mathcal{P}; \mathcal{B}; I)$. We
make a partition $\mathcal{G}'$ of $\mathcal{B}$ such that the parallel affine hyperplanes belong to the same
group in $\mathcal{G}'$. Then $(\mathcal{B}; \mathcal{G}';\mathcal{P})$ becomes a $GDD(n;m; k; 0; \lambda_2)$, where $n = q- 1$, $m = \frac{q^d-1}{q-1}$
is the number of $d-1$ dimensional subspaces in $V$, $k = q^{d-1}$ is the number of non-zero
vectors in an affine hyperplane, and $\lambda_2 = q^{d-2}$ is the number of non-zero vectors in the
intersection of two given non-parallel affine hyperplanes.

This shows that $\mathcal{D}(d; q):= (\mathcal{P};\mathcal{G}; \mathcal{B})$ is a $GDDDP(n;m; k; 0; \lambda_2)$, where $n; m; k; \lambda_2$
are given above. We denote $\Gamma(d, q) := \Gamma(\mathcal{D}(d; q))$ as the point-block incidence graph of $\mathcal{D}(d; q)$. It is clear that the general linear group $GL(d, q)$ acts as a group of automorphism of the graph $\Gamma(d, q)$.

By \cite[Proposition 4]{QDK-2016},  the point-block incidence graphs $\Gamma(d, q)$ of classical group divisible
designs with the dual property $\mathcal{D}(d; q)$ are 2-arc-transitive dihedrants.

For any $r|q-1$, let $N\leq Z=Z(GL(d,q))\cong \mathbb{Z}_{q-1}$ and $|N|=(q-1)/r$. Let
$\Gamma(d, q, r)$ be the quotient graph of $\Gamma(d,q)$ induced by $N$.

We use  $X(3, 2)$ to denote the graph with the vertex set
$\{i|i\in \mathbb{Z}_{14}\}\cup \{i'|i'\in \mathbb{Z}_{14}\}$ where $i$ and $j'$ are adjacent if
and only if $j'-i\in \{0,1,9,11\}$.

\bigskip

\bigskip

\section{Proof of Theorem \ref{th-quot-miniblock-1}}


In this section, we give a reduction result for the family of $(G,2)$-arc-transitive  bicirculants.

Let $\Gamma$ be a $(G,2)$-arc-transitive  bicirculant over a cyclic subgroup $H$ of $G$. Then the first lemma characterizes normal subgroups of $G$ which have  at least $3$ orbits and also presents a
relationship between the two orbits of  $H$  and the set of $N$-orbits.

\begin{lemma}\label{quot-orbit-01}
Let $\Gamma$ be a connected $(G,2)$-arc-transitive  bicirculant   over   the    cyclic subgroup $H$ of $G$.
Let $H_0$ and $H_1$ be the two orbits of $H$ on $V(\Gamma)$. Let $N$ be a  normal subgroup of $G$ with  at least $3$ orbits and let $B$ be an  $N$-orbit.   Then $N$ is regular on each orbit, and either

\begin{enumerate}[{\rm (1)}]
\item  $B\subset H_i$ for some $H_i$, and  $N\leq H$       is cyclic; or
\item  $|B\cap H_0|=|B\cap H_1|$, and  $N$       has a cyclic index $2$ normal subgroup $H\cap N$.

\end{enumerate}

\end{lemma}
\proof Since $\Gamma$ is a non-complete $(G,2)$-arc-transitive graph, it follows that $\Gamma$ is $(G,2)$-distance-transitive with girth at least 4. Particularly, $\Gamma\ncong \K_{m[b]}$ for any  $m\geq 3$ and $b\geq 2$.  Since the normal subgroup $N$ has at least three orbits on the vertex set, it follows from  Lemma \ref{dt-quotient12} that
$N$ is regular on each orbit, and $N$ is the kernel of the action of $G$ on the set of $N$-orbits.

Let $\mathcal{B}$ be the set of $N$-orbits.  Then by Lemma \ref{quo-1},
for each $B\in \mathcal{B}$, either $B\subset H_i$ for some $H_i$ or $|B\cap H_0|=|B\cap H_1|$, where $H_0$ and $H_1$ are the two orbits of $H$ on $V(\Gamma)$.
Suppose first that $B\subset H_i$ for some $H_i$. Then $HN/N\cong H/(H\cap N)$ is regular on $\mathcal{B}\cap H_i$. Consequently
$H\cap N$ is regular on $B$, and so   $|H\cap N|=|B|$. As $N$ is regular on each orbit, it follows that
$|H\cap N|=|B|=|N|$,
hence we have $H\cap N=N$. Thus  $N\leq H$ is cyclic, and (1) holds.
Next suppose that   $|B\cap H_0|=|B\cap H_1|$. Then $HN/N\cong H/(H\cap N)$ is regular on $\mathcal{B}$, and so
$H\cap N$ is semiregular on $B$ with two orbits. Thus  $H\cap N$ is a cyclic index 2 normal subgroup of $N$, so that (2) holds.
\qed


\medskip
We give a reduction result for the family of $(G,2)$-arc-transitive  bicirculants.

\begin{theo}\label{bic-reduction-th2}
Let $\Gamma$ be a connected $(G,2)$-arc-transitive  bicirculant   over   the    cyclic subgroup $H$ of $G$.
Let $H_0$ and $H_1$ be the two $H$-orbits on $V(\Gamma)$. Let $N$ be a  normal subgroup of $G$ maximal with respect to having at least $3$ orbits. Then
$\Gamma$ is a cover of $\Gamma_N$ which
 is $(G/N,2)$-arc-transitive, and $G/N$ is faithful and either quasiprimitive or bi-quasiprimitive on $V(\Gamma_N)$.
Moreover, one of the following holds.

\begin{enumerate}[$(1)$]
\item each $N$-orbit is a subset of  either $H_0$ or $H_1$,   $N\leq H$       is cyclic, and $\Gamma_N$ is a bicirculant of $H/N$;

 \begin{itemize}
\item[(1.1)]  $G/N$ is quasiprimitive on $V(\Gamma_N)$ and $\Gamma_N$ is one of the following graphs:

\begin{enumerate}[$(a)$]

\item  $\K_{2,2}$;

\item Petersen graph;

\item  folded $5$-cube;

\item  $\K_n$ where $n\geq 3$;

\item $ \K_{n,n}-n\K_2$ where $n\geqslant 3$.
\end{enumerate}

\item[(1.2)] $G/N$ is bi-quasiprimitive on $V(\Gamma_N)$ and $\Gamma_N$ is one of the following graphs:

\begin{enumerate}[$(a)$]
\item  $B(H(11))$ and $B'(H(11))$;

\item $\K_{n,n}$ where $n\geqslant 2$;

\item  $ \K_{n,n}-n\K_2$ where $n\geqslant 3$;


\item $B(\PG(d-1,q))$ and $B'(\PG(d-1,q))$, where $d\geq 3$, $q$ is a prime power.

\end{enumerate}

\end{itemize}

\item each $N$-orbit intersects both $H_0$ and $H_1$ non-trivially, and  $N$       has a cyclic index $2$ normal subgroup $H\cap N$, and $\Gamma_N$ is one of the following circulants:

\begin{enumerate}[$(a)$]
\item the cycle $C_n$ where $n\geq 3$;

\item   $\K_n$ where $n\geq 3$;

\item  $\K_{n,n}$ where $n\geq 3$;

\item  $\K_{n,n}-n\K_2$ where  $n\geq 5$ is odd.

\end{enumerate}

\end{enumerate}

\end{theo}
\begin{proof} Since $\Gamma$ is a $(G,2)$-arc-transitive graph,  it follows that $\Gamma$ is  a $(G,2)$-distance-transitive graph and also  $\Gamma\ncong \K_{m[f]}$ for any  $m\geq 3$ and $f\geq 2$. Since $N$ is a  normal subgroup of $G$ maximal with respect to having at least $3$ orbits,   by Lemma \ref{dt-quotient12},   $N$ is regular on each orbit  and $\Gamma$ is a cover of $\Gamma_N$.
Moreover, all normal subgroups of the quotient group $G/N$ are transitive or have exactly two orbits on $V(\Gamma_N)$. Thus $G/N$ is quasiprimitive or bi-quasiprimitive on $V(\Gamma_N)$.

Let $\mathcal{B}=\{B_1,\ldots,B_t\}$ be the set of $N$-orbits where  $t\geqslant 3$. Since $\Gamma$ is $(G,2)$-arc-transitive, it is easy to show  that $\Gamma_N$ is a $(G/N,2)$-arc-transitive graph, and  each induced subgraph $[B_i]$ is  edge-less.
It follows from  Lemma \ref{quot-orbit-01} that either

\begin{enumerate}[{\rm (i)}]
\item  all elements of $\mathcal{B}$ are subsets of $H_0$ or $H_1$ and $N\leqslant H$ is a cyclic group; or
\item  $B\cap H_0\neq \varnothing$ and $B\cap H_1\neq \varnothing$ for each   $B\in \mathcal{B}$, and  $N$       has a cyclic index $2$ normal subgroup $H\cap N$.

\end{enumerate}

Suppose first that (i) holds. Then
for each $B\in \mathcal{B}$, either $B\subset H_0$ or $B\subset H_1$.
Moreover, by the proof of Lemma \ref{quot-orbit-01}, we know that $HN/N\cong H/(H\cap N)\cong H/N$ is regular on $\mathcal{B}\cap H_i$.  Thus
$\Gamma_N$ is a bicirculant over $ H/N$.

Assume  that $G/N$ is quasiprimitive on $V(\Gamma_N)$. Then as   $\Gamma_{N}$ is a $(G/N,2)$-arc-transitive bicirculant, it follows that $\Gamma_{N}$  is one of the  graphs in  \cite[Proposition 4.2]{DGJ-2019}.
Note that the following graphs: $\H(2,4)$ and its complement, $\K_{n[2]}$ with $n\geq 3$, complement of the folded $5$-cube  and  complement of the Petersen graph   have girth 3. Thus these graphs are  not 2-arc-transitive.
Therefore,  $\Gamma_{N}$  is one of the  graphs in Table \ref{table-at-quasi}.

\begin{table}[h]\caption{$(G/N,2)$-arc-transitive and $G/N$-quasiprimitive  bicirculants }\label{table-at-quasi}
\begin{tabular}{llll}
\hline
  & Graphs   \\
\hline\hline
 1  &  $\K_{2,2}$;\\
 \hline
 2  &   Petersen graph and $A_5\leqslant G/N\leqslant S_5$;          \\
\hline
3   &    folded $5$-cube and $G/N$ is a rank $3$ subgroup of $\AGL(4,2)$;  \\
\hline
4&  $\K_{2n}$ where $n\geq 2$;  \\
\hline
5& $ \K_{n,n}-n\K_2$ where $n\geq 3$.\\
\hline
\end{tabular}

\end{table}

Now suppose that  $G/N$ is bi-quasiprimitive on $V(\Gamma_N)$.
Then $G/N$ has a minimal normal subgroup $M/N$ that has exactly two orbits on $V(\Gamma_N)$.
Since  $\Gamma_N$ is $G/N$-arc-transitive and connected,
each $M/N$-orbit contains no edge of $\Gamma_N$.
Thus $\Gamma_N$ is a bipartite graph, and  the two $M/N$-orbits form the two bipartite halves of $\Gamma_N$.
In particular,   all intransitive  normal subgroups of $G/N$ have the same orbits.

Assume that  the two orbits of the cyclic subgroup $H/N$ are the two bipartite halves of $\Gamma_N$. Then $\Gamma_N$ is one of the  graphs
in  \cite[Proposition 5.1]{DGJ-2019}.

Let $\Gamma_N$ be the graph $G(2p,r)$ of \cite[Proposition 5.1]{DGJ-2019}.
Then $V(\Gamma_N)$ consists of the elements $i$ and $i'$ for $i\in\mathbb{Z}_p$. Let
\begin{align*}
\tau: & \  V(\Gamma_N)\mapsto V(\Gamma_N),i\mapsto i+1,i'\mapsto (i+1)',\\
\sigma: & \  V(\Gamma_N)\mapsto V(\Gamma_N),i\mapsto (-i)', i'\mapsto -i.
\end{align*}

Then $\tau$ is an automorphism of $\Gamma_N$ of order  $p$ consisting of two $p$-cycles,
and $\sigma$ is an automorphism of $\Gamma_N$ of order  $2$ swapping the two orbits of $\tau$.
Moreover,  $\sigma \tau \sigma =\tau^{-1}$, and $\langle \sigma,\rho\rangle \cong D_{2p}$ is a dihedral group of order $2p$ that acts regularly on the vertex set.
Thus $\Gamma_N$ is a Cayley graph of $D_{2p}$.
Since $\Gamma_N$ is $(G/N,2)$-arc-transitive, it follows that $\Gamma_N$ is a graph of \cite[Theorem 1.1]{DMM-2008}.
Moreover,  \cite[Theorem 1.1]{DMM-2008} indicates that  $\Gamma_N$ is 2-arc-transitive if and only if $p=11,r=5$ and $\Gamma_N=G(22,5)=B(H(11))$.

Thus  $\Gamma_N$ is one of the  graphs
in Table \ref{table-at-biquasi-1}.

\begin{table}[h]\caption{$(G/N,2)$-arc-transitive and $G/N$-bi-quasiprimitive  bicirculants (I)}\label{table-at-biquasi-1}
\begin{tabular}{llll}
\hline
  & Graphs   \\
\hline\hline
 1  &    $\K_{n,n}$ where $n\geqslant 2$;  \\
 \hline
 2  &   $ \K_{n,n}-n\K_2$ where $n\geqslant 3$;          \\
\hline
3   & $B(\PG(d-1,q))$ and $B'(\PG(d-1,q))$, where $d\geq 3$, $q$ is a prime power;\\
\hline
 4  &   $B(H(11))$ and $B'(H(11))$.          \\
\hline
\end{tabular}

\end{table}

Assume  that $H/N$ has two orbits of size $n$ and these are not the bipartite halves of $\Gamma_N$.
Then by  \cite[Proposition 5.2]{DGJ-2019}, $n$ is even and  $\Gamma_N$ is one of the graphs in Table \ref{table-at-biquasi-2}.

\begin{table}[h]\caption{$(G/N,2)$-arc-transitive and $G/N$-bi-quasiprimitive  bicirculants (II)}\label{table-at-biquasi-2}
\begin{tabular}{llll}
\hline
  & Graphs   \\
\hline\hline
 1  &    $\K_{n,n}$, where $n\geq 2$ is even;  \\
 \hline
 2  &    $\K_{n,n}-n\K_2$, where $n\geq 3$ is even;          \\
\hline
 3  &    $B(\PG(d-1,q)$ and $B'(\PG(d-1,q))$ for even $d\geqslant 3$ and odd $q$.          \\
\hline
\end{tabular}

\end{table}

Finally,  suppose that (ii) holds, that is,
$B\cap H_0\neq \varnothing$ and $B\cap H_1\neq \varnothing$ for some $B\in \mathcal{B}$.
Then   $B'\cap H_0\neq \varnothing$ and $B'\cap H_1\neq \varnothing$ for every $B'\in \mathcal{B}$.
Hence, by  \cite[Lemma 3.2]{DGJ-2019},  $\Gamma_N$  is isomorphic to a  circulant of $HN/N$, and also Lemma \ref{quot-orbit-01}
says that $N$       has a cyclic index $2$ normal subgroup $H\cap N$.
Moreover, the 2-arc-transitivity of  $\Gamma_{N}$ indicates that $\Gamma_{N}$ is a graph in Theorem \ref{circulant-2arctrans-1},
and so  $\Gamma_{N}$  is one of the following graphs:

\begin{itemize}
\item[(1)]  $\K_n$ where $n\geq 3$;

\item[(2)]   $\K_{n,n}$ where  $n\geq 3$;

\item[(3)] $\K_{n,n}-n\K_2$ where  $n\geq 5$ is odd;

\item[(4)] the cycle $C_n$ where $n\geq 3$.

\end{itemize}

We conclude the proof.
\end{proof}

\bigskip

\bigskip

\section{Cyclic Covers }

This subsection is devoted to finding the cyclic covers of the basic $2$-arc-transitive bicirculants in case (1) of  Theorem \ref{bic-reduction-th2}.

\begin{lemma}\label{cyc-kn}
Suppose that  $\Gamma $ is a  $2$-arc-transitive bicirculant arising as a regular cyclic cover of the complete graph  $\K_{n}$ where $n\geq 3$. Then $\Gamma$
is either
\begin{itemize}
\item[(1)]  $\K_{n,n}-n\K_2$; or

\item[(2)]   $ X_1(4,q)$ where $q\equiv 3\pmod{4}$ and  $q=n-1$.

\end{itemize}

\end{lemma}
\proof  Since $\Gamma $ is a  $2$-arc-transitive cyclic cover of the complete graph $\K_n$ where $n\geq 3$, it follows from
\cite[Theorem 1.1]{DMW-1998} that $\Gamma$
is either  $\K_{n,n}-n\K_2$ or $ X_1(4,q)$ where $q=n-1$.
Moreover, for the graph $ X_1(4,q)$,  Lemma \ref{lem-cov-comp-1} says that    it is a 2-arc-transitive bicirculant over a cyclic group of order  $2(q+1)$. For the graph $\K_{n,n}-n\K_2$, its automorphism group is  $ S_n\times S_2$ which has a cyclic subgroup of order $n$, and this cyclic subgroup  is semiregular on the vertex set, so
$\K_{n,n}-n\K_2$ is also a 2-arc-transitive bicirculant.
 \qed

\begin{lemma}\label{cyccov-petersen}

Suppose that  $\Gamma $ is a  $2$-arc-transitive bicirculant arising as a regular cyclic cover of the Petersen  graph. Then $\Gamma$
is either the  Desargues graph or the dodecahedron graph.
\end{lemma}
\proof Since  $\Gamma $ is a  $2$-arc-transitive graph arising as a regular cyclic   cover of the Petersen  graph, it follows from
 \cite[Theorem 4.4]{FengKwak-10p} that
$\Gamma$ is either the dodecahedron graph   or the Desargues graph (the standard double cover of the Petersen  graph).

Moreover, the dodecahedron graph is the generalized Petersen graph $GP(10,2)$ which  has automorphism group $A_5\times S_2$,
and the Desargues graph
is the generalized Petersen graph $GP(10,3)$ which has automorphism group $S_5\times S_2$ (see for instance \cite[p.217]{FGW-1971}). Thus
both of these two graphs are  $2$-arc-transitive bicirculants.
 \qed

\begin{lemma}\label{cyc-knn}


Let  $\Gamma \cong \K_{n,n}$ where $n\geq 3$.
Then  $\Gamma$ does not have  $2$-arc-transitive  regular cyclic   covers as bicirculants.


\end{lemma}
\proof  Suppose to the contrary that  $\Gamma\cong \K_{n,n}$ where $n\geq 3$, has   a 2-arc-transitive regular cyclic cover $\overline{\Gamma}$ which is  a bicirculant.
Then   by Lemma \ref{2at-dih-lem1}, we can assume  that the covering transformation group
$N$ is a cyclic group $\mathbb{Z}_p$ for some prime $p$.

By Theorem \ref{bic-reduction-th2},
either $\Gamma$ is a  bicirculant over  a cyclic group of order $n$, and  the two bipartite halves of $\Gamma$ are orbits of this group or
$\Gamma$ is a circulant over a cyclic group $\mathbb{Z}$ of order $2n$. If the latter holds, then
$\Gamma$ is a bicirculant over the   index 2 normal cyclic subgroup $\mathbb{Z}'$ of  $\mathbb{Z}$ of order $n$, and
the two bipartite halves of $\Gamma$ are $\mathbb{Z}'$-orbits.
Thus the graph  $\Gamma$ always can be considered as  a  bicirculant over  a cyclic group of order $n$, and say this cyclic group $H$,
and
the two bipartite halves of $\Gamma$ are $H$-orbits. Let the two bipartite halves of  $\Gamma$
be $U$ and $W$.

Since $\overline{\Gamma}$ is a cover of  $\Gamma$ which  is a bipartite graph, and since each $N$-orbit does not contain any edge of $\overline{\Gamma}$, it follows that $\overline{\Gamma}$ is also a  bipartite graph.
And let $\overline{U}$ and $\overline{W}$ be the two bipartite halves of  $\overline{\Gamma}$ with respect   to
$U$ and $W$, respectively.

In light of the fact that  $\Gamma$ is $2$-arc-transitive, $\Aut(\Gamma)$ has a subgroup  $G$    that acts transitively on the set of 2-arcs  of $\Gamma$, and let $G^+\leq G$  be the corresponding index 2 normal subgroup of $G$ fixing $U$ and $W$ setwisely. Let $G^+_U$  be the kernel of the action of $G^+$ on $U$
and let  $G^+_W$ be the kernel of the action of $G^+$ on  $W$. Then $G^+_U$   fixes the set $U$   pointwisely and $G^+_W$  fixes the set $W$ pointwisely. Moreover, $G^+_U\cong G^+_W$ and $G^+_U\cap G^+_W=1$.

Since $\Gamma$ is $(G,2)$-arc-transitive, it follows that for each vertex $w\in  W$, the vertex stabilizer $G_w = G^+_w$ is 2-transitive on $\Gamma(w)=U$.
Note that  $G^+_w\cap G^+_U$  is the kernel of the action of $G^+_w$ on $U$. Hence
 $G^+_wG^+_U/G^+_U\cong G^+_w/(G^+_w\cap G^+_U)$ is 2-transitive on $U$. Thus $G^+/G^+_U$ acts 2-transitively on $U$.
By a similar argument,
we can show that $G^+/G^+_W$ acts 2-transitively on $W$.

A result of Burnside \cite{Burnside} shows that every 2-transitive group has a
unique minimal normal subgroup that is either elementary abelian or
nonabelian simple.
The fact  $G^+_U\cap G^+_W=1$ leads to  that
$G^+_U\cong G^+_U/(G^+_U\cap G^+_W)  \cong G^+_U G^+_W/G^+_W$  is a normal subgroup of $G^+/G^+_W$, and so $G^+_U$ contains a normal subgroup $T_U\cong  soc(G^+/G^+_W)$. Similarly, $G^+_W$ contains a normal subgroup $T_W\cong soc(G^+/G^+_U)$. Moreover, it is easy to see that $T_U\cong T_W$.

Suppose that $\overline{\Gamma} = \Gamma \times_\psi \mathbb{Z}_p$  and such that $G$ lifts. Denote the lifts of $G$ and $G^+$ by $\overline{G}$ and $\overline{G^+}$,  respectively. We argue the proof depending on whether $T_U$ is nonabelian simple or elementary abelian.

Assume that $T_U$ is a nonabelian simple group.
Then by \cite[Theorem 1.2]{XD-2018}, $n\neq 28$, and so   $T_U$ and $T_W$ act 2-transitively on $W$ and $U$, respectively. Let  $T = T_U \times T_W$.
Assume that  the group $T$ is lifted along the covering projection to the group $\overline{T}\leq   \Aut(\overline{\Gamma})$. Then $\overline{T}/N\cong T$. Since $N\cong \mathbb{Z}_p$ is abelian, it follows that $N\leq C_{\overline{T}}(N)$ and $C_{\overline{T}}(N)$ is a normal subgroup of $\overline{T}=N_{\overline{T}}(N)$. Thus
$(\overline{T}/N)/(C_{\overline{T}}(N)/N)\cong \overline{T}/C_{\overline{T}}(N)=N_{\overline{T}}(N)/C_{\overline{T}}(N)\leq \Aut(N)\cong \mathbb{Z}_{p-1}$,
and so  $\overline{T}/N=T$ has an abelian quotient.

However, as $T_U$ and $T_W$ are nonabelian simple groups, it follows that the only nontrivial normal  proper subgroups of $T$ are $T_U$ and $T_W$. Therefore each nontrivial quotient of $T$ must be nonabelian,
and hence $\overline{T}/N=C_{\overline{T}}(N)/N$,
that is, $\overline{T}=C_{\overline{T}}(N)$. Consequently,
the group $T$ lifts to the group $\overline{T}$ centralizing $N$. Then it follows from Lemma \ref{lem-cent-1} that in the action of $T$ on the set of closed walks of $\Gamma$, the voltages of closed
walks of $\Gamma$ are the same  in  each orbit of $T$. Moreover,
since   $T_U$ acts 2-transitively on $W$   and $T_W$ acts 2-transitively on  $U$,
it follows that   $T$ acts transitively on the set of all 4-cycles in $\Gamma$, and so all the  4-cycles
of $\Gamma$ have the same voltage. Pick three vertices $u_0,u_1,u_2\in  U$ and two vertices $w_0,w_1\in  W$. Then  in the subgraph induced by $\{u_0,u_1,u_2,w_0,w_1 \}$, we have three 4-cycles:
$C_0 = (u_0,w_0,u_1,w_1,u_0), C_1 = (u_0,w_0,u_2,w_1,u_0)$ and $	C_2 = (u_1,w_1,u_2,w_0,u_1)$.
Moreover, $C_1 $ has arcs $\{(u_0,w_0),(w_0,u_2),(u_2,w_1),$ $(w_1,u_0)\}$
and $C_2 $ has arcs $\{(u_1,w_1),(w_1,u_2),(u_2,w_0),$ $(w_0,u_1)\}$, and further
$\psi((w_0,u_2))+\psi((u_2,w_0))=0$ and $\psi((u_2,w_1))+\psi((w_1,u_2))=0$.
Thus $\psi(C_0)=\psi(C_1) + \psi(C_2) = \phi(C_1+C_2)$,
and so  $\psi(C_2) = \psi(C_0) - \psi(C_1) = 0$. Therefore the 4-cycles of $\Gamma$ have trivial voltage. Since 4-cycles span the cycle space of $\Gamma$, it follows from Lemma \ref{lem-vanish}  that the graph $\overline{\Gamma}$ is disconnected, a contradiction.

It remains to consider the case that  $T_U$ is an elementary abelian group, and say $T_U\cong \mathbb{Z}^t_q$, where $q$ is a prime and $n = q^t$ for some $t$.
Then $\Gamma$ has $2q^t$ vertices.
Since $G^+/G^+_W$ acts 2-transitively on $W$ and $T_U\cong  soc(G^+/G^+_W)$,  it follows that $T_U$ is regular on $W$, and hence by
 \cite[Proposition 4.4]{Wielandt-book},  $T_U=C_{G^+/G^+_W}(T_U)$.
Thus the quotient group $(G^+/G^+_U)/T_U=(G^+/G^+_U)/C_{G^+/G^+_W}(T_U)$ is a subgroup of $\Aut(T_U)\cong GL(t,q)$, and so
 $G^+/G^+_W$ is a subgroup of $AGL(t,q)=\mathbb{Z}^t_q:GL(t,q)$.

Since   $\Gamma$ is  a bicirculant over  a cyclic group $H$ of order $n=q^t$ and
the two bipartite halves are $H$-orbits,
the group $G^+/G^+_W$ contains an element of order $q^t$, and by \cite[Lemma 4.1 (ii)]{DMM-2008},  we have $n = q$ or $n = 4$.

If   $n = q$ is a prime, then by \cite[Theorem 1.2]{XD-2018}, $\Gamma$ is the graph $X(1,p)$ which is not a bicirculant, a contradiction.
Assume  that $n=4$ and $q=t = 2$. Then
$G^+/G^+_W$ is a subgroup of $AGL(2,2)=\mathbb{Z}^2_2:GL(2,2)$.
Note that  $G^+/G^+_U$ acts 2-transitively on $W$ and  contains an element of order $4$. However, by   \cite[Theorem 3.1]{SLZ-2014} (see also \cite{Jones-2002},\cite{LCH-abelianregular-2003}),  such a
2-transitive group does not exist.
We conclude the proof.
 \qed

\medskip
The following lemma investigates  $2$-arc-transitive bicirculants arising as regular cyclic covers of the graph $B'(H_{11})$, and
the argument  of    $2$-arc-transitive bicirculants arising as regular cyclic covers of the graph  $B(H_{11})$ is similar.

\begin{lemma}\label{cyc-H11}


There are no $2$-arc-transitive bicirculants arising as regular cyclic covers of the graphs  $B'(H_{11})$ or $B(H_{11})$.
\end{lemma}
\proof By Lemma \ref{2at-dih-lem1}, it suffices to show that for any   prime $p$, there are  no such regular $\mathbb{Z}_p$-covers.
Suppose to the contrary that  $\Gamma$ is a
regular $\mathbb{Z}_p$-cover of $B'(H_{11})$ for some prime $p$. Let $\Sigma=B'(H_{11})$ and $N=\mathbb{Z}_p$.
Then  $\Sigma$ is a bipartite graph. Let the two bipartite halves of  $\Sigma$
be $U$ and $W$, and let $G^+$ be the setwise stabilizer of the bipartite halves in $G:=\Aut(\Sigma)$. Then $G^+\cong PSL(2,11)$ is an  normal subgroup of
$G$ such that $|G:G^+|=2$. Let $\overline{G^+}$ be the lift of $G^+$ in $\Aut(\Gamma)$. Then $\overline{G^+}/N\cong G^+$.

In light of the fact that  $N=\mathbb{Z}_p$, we have $N\unlhd C_{\overline{G^+}}(N)$. Moreover,
$C_{\overline{G^+}}(N)/N$ is a normal subgroup of $N_{\overline{G^+}}(N)/N=\overline{G^+}/N=G^+$. As $G^+\cong PSL(2,11)$ is a nonabelian simple group, it follows that
$C_{\overline{G^+}}(N)/N=N_{\overline{G^+}}(N)/N$ or 1. Consequently
$C_{\overline{G^+}}(N)=N_{\overline{G^+}}(N)$ or $N$. Suppose that $C_{\overline{G^+}}(N)=N$. Then
$\overline{G^+}/N=\overline{G^+}/C_{\overline{G^+}}(N)\leq \Aut(N)\cong \mathbb{Z}_{p-1}$, that is,
$G^+=\overline{G^+}/N\leq  \mathbb{Z}_{p-1}$, which is impossible. Thus $C_{\overline{G^+}}(N)=N_{\overline{G^+}}(N)$.

Since $G^+$ is a nonabelian simple group, we have $G^+=(G^+)'$. Thus  $\overline{G^+}'N/N=(\overline{G^+}/N)'=(G^+)'=G^+=\overline{G^+}/N$, and so
$\overline{G^+}'N=\overline{G^+}$. Since $N=\mathbb{Z}_p$ where $p$ is a prime, we have $\overline{G^+}'\cap N=1$ or $N$.

Suppose that $\overline{G^+}'\cap N=N$.  Then since $C_{\overline{G^+}}(N)=N_{\overline{G^+}}(N)=\overline{G^+}$, we have $N\leq \overline{G^+}'\cap Z(\overline{G^+})$, and so $N$ is contained in the Schur multiplier of $G^+$.
The Schur multiplier of $G^+=PSL(2,11)$ is $\mathbb{Z}_2$, and so $N=\mathbb{Z}_2$. It follows that $\overline{G^+}=SL(2,11)$. Since
$SL(2,11)$ has only one involution, it follows that $A_5$  is not a subgroup of $\overline{G^+}$, contradicts that
the vertex stabilizer of $\overline{G^+}$ is isomorphic to $A_5\cong PSL(2,5)$. Hence $\overline{G^+}'\cap N\neq N$.

Thus  $\overline{G^+}'\cap N=1$. Then $\overline{G^+}=\overline{G^+}'\times N$ and $\overline{G^+}'\cong \overline{G^+}'/(\overline{G^+}'\cap N) \cong \overline{G^+}'N/N\cong \overline{G^+}/N\cong G^+\cong PSL(2,11)$.
In light of the fact that  $\Gamma$ is a cover of  $\Sigma$ which  is a bipartite graph, it follows that $\Gamma$ is also a  bipartite graph.
Let $\overline{U}$ and $\overline{W}$ be the two bipartite halves of  $\Gamma$ with respect  to
$U$ and $W$, respectively.
Suppose first that $\overline{G^+}'$ is transitive on only one bipartite half, say $\overline{U}$.  Take a fibre $\overline{B}$
in $\overline{U}$ and let $\overline{u}\in \overline{B}$.
Note that the vertex stabilizer $(\overline{G^+}')_{\overline{u}}$
is a subgroup of $(\overline{G^+}')_{\overline{B}}\cong PSL(2,5)$.
Since  $|\overline{B}|=p$  and  $PSL(2,5)$ does not contain subgroups of index 2 or 3, it follows that
$p=5$, this implies that $(\overline{G^+}')_{\overline{u}}\cong A_4$. As
$N_{\overline{G^+}'}((\overline{G^+}')_{\overline{u}})=(\overline{G^+}')_{\overline{u}}$, $\overline{u}$
is the only fixed vertex of $(\overline{G^+}')_{\overline{u}}$ in $\overline{U}$. Since
$(\overline{G^+}')_{\overline{B}}=(\overline{G^+}')_{\overline{u}}\times N$, it follows that
$(\overline{G^+}')_{\overline{u}}\leq (\overline{G^+}')_{\overline{B}}$ centralizes $N$. Since $N$ is transitive on $\overline{B}$,
$(\overline{G^+}')_{\overline{u}}$ fixes all 5 vertices in $\overline{B}$, which is a contradiction.

Suppose next that $\overline{G^+}'$ is transitive on both  the two bipartite halves  $\overline{U}$ and $\overline{W}$.
Then $\overline{G^+}'$ has $p$ orbits of size 11 on both $\overline{U}$ and $\overline{W}$. It follows that for each
vertex $\overline{u}$ of $\Gamma$, the vertex stabilizer $(\overline{G^+}')_{\overline{u}}$ and $(\overline{G^+})_{\overline{u}}$
are the same group. Thus $(\overline{G^+})_{\overline{u}}=(\overline{G^+}')_{\overline{u}}\leq \overline{G^+}'$.
Let $\{\overline{u},\overline{w}\}$ be an edge of $\Gamma$. Then by Lemma 2.4 of  \cite{DX-2000}, $\Gamma$ is isomorphic to a bi-coset graph
$Cos(\overline{G^+},(\overline{G^+})_{\overline{u}},(\overline{G^+})_{\overline{w}},D)$ where $D=(\overline{G^+})_{\overline{u}}(\overline{G^+})_{\overline{w}}$.
Since $(\overline{G^+})_{\overline{u}}$ and $(\overline{G^+})_{\overline{w}}$
are contained in $\overline{G^+}'$, it follows that $\langle D^{-1}D\rangle=\langle (\overline{G^+})_{\overline{u}},(\overline{G^+})_{\overline{w}}\rangle$
is not equal to $\overline{G^+}$. Thus by Lemma 2.3  of  \cite{DX-2000}, $\Gamma$ is disconnected, a contradiction. \qed

\begin{lemma}\label{cyccov-folded5cube}
There are no $2$-arc-transitive bicirculants arising as regular cyclic covers of the folded $5$-cube.

\end{lemma}
\proof Let $\Gamma$ be the folded $5$-cube. Then $\Gamma$ has valency 5. Suppose that
$\overline{\Gamma}$ is a $2$-arc-transitive bicirculant arising as a regular cyclic cover of $\Gamma$. Then
$\overline{\Gamma}$ is a symmetric valency 5 bicirculant. Note that the family of arc-transitive bicirculants of valency 5 has been classified in  \cite{AHK-2015,AHKOS-2015}.

Let    $H:=\mathbb{Z}_{24}$ be an additive group. Define the bicirculant  $BC_{24}[\varnothing,\{0,1,3,11,20\},\varnothing]$
to have vertex set the union of the left part $H_0=\{h_0|h\in H\}$ and the right part $H_1=\{h_1|h\in H\}$, and edge set   $\{\{h_0,(h+m)_1\}|m \in \{0,1,3,11,20\}\}$.

Since $|V(\Gamma)|=16$ and $\overline{\Gamma}$ is a  cover of $\Gamma$,
it follows that $16$ divides $|V(\overline{\Gamma})|$. By inspecting the candidates in \cite{AHK-2015,AHKOS-2015},    $\overline{\Gamma}$ is one of the following graphs:

 \begin{itemize}
\item[(1)] folded $5$-cube;

\item[(2)]  $BC_{24}[\varnothing,\{0,1,3,11,20\},\varnothing]$;

\item[(3)] $\Cay(D_{2n},\{b,ba,ba^{r+1},ba^{r^2+r+1},ba^{r^3+r^2+r+1}\})$, where   $D_{2n}=\langle a,b|a^n=b^2=(ba)^2=1 \rangle$,   $r \in \mathbb{Z}_n^*$ such that  $r^4+r^3+r^2+r+1 \equiv 0 \pmod n$ and $8|n$.
\end{itemize}

If  $\overline{\Gamma}=\Cay(D_{2n},\{b,ba,ba^{r+1},ba^{r^2+r+1},ba^{r^3+r^2+r+1}\})$, then $\overline{\Gamma}$ is a dihedrant, and so it is listed in \cite[Theorem 1.1]{DMM-2008}. However, by inspecting the graphs in \cite[Theorem 1.1]{DMM-2008}, $\overline{\Gamma}$ does not exist.

Now suppose that  $\overline{\Gamma}=BC_{24}[\varnothing,\{0,1,3,11,20\},\varnothing]$.  Then $\overline{\Gamma}$ is bipartite, and say the two bipartite halves
being $U$ and $W$.  Let $\mathcal{B}$ be the block system of $V(\overline{\Gamma})$
such that $\overline{\Gamma}_{\mathcal{B}}\cong \Gamma$. Let $B\in \mathcal{B}$. If $B$ is contained in $U$ or $W$, then
$\overline{\Gamma}_{\mathcal{B}}$ is bipartite. However,  $\Gamma$  is not a bipartite graph as it has girth 5,  a contradiction. Thus $|B\cap U|=|B\cap W|$, and so
$|C\cap U|=|C\cap W|$ for every $C\in \mathcal{B}$. Therefore
$|B\cap U|$ divides 24 and also  $|B\cap U|\times 16=24$, which is impossible. \qed

\begin{lemma}\label{cyc-pg}

\begin{itemize}
\item[(1)]  The graph  $B(\PG(d-1,q))$ does not have  $2$-arc-transitive  regular cyclic   covers as bicirculants  where $d\geq 3$ and $q$ is a prime power.

\item[(2)]   Suppose that $\overline{\Gamma}$ is a   $2$-arc-transitive  regular cyclic   cover of   $B'(\PG(d-1,q))$ as bicirculants  where $d\geq 3$ and $q$ is a prime power.
Then $\overline{\Gamma}$ is either $X(3,2)$ or the graph $\Gamma(d, q, r)$ where $r|q-1$.

\end{itemize}

\end{lemma}
\proof Let $\Gamma$ be a graph that is either $B(\PG(d-1,q))$ or $B'(\PG(d-1,q))$ where $d\geq 3$ and $q$ is a prime power.
Then $\Gamma$ is a bipartite graph, and we can assume   the two bipartite halves of  $\Gamma$
are $U$, the set of 1-dimensional subspaces, and $W$, the set of $(d-1)$-dimensional subspaces.


Suppose that  $\Gamma$ has   a 2-arc-transitive regular cyclic cover $\overline{\Gamma}$ which is  a bicirculant. Then by Lemma \ref{2at-dih-lem1}, we can assume  that the covering transformation group
$N$ is a cyclic group $\mathbb{Z}_p$ for some prime $p$.
Moreover, suppose that $\overline{\Gamma}:= \Gamma \times_\psi N$.

Since $\overline{\Gamma}$ is a cover of  $\Gamma$ which  is a bipartite graph, it follows that $\overline{\Gamma}$ is also a  bipartite graph.
And let $\overline{U}$ and $\overline{W}$ be the two bipartite halves of  $\overline{\Gamma}$ related  to
$U$ and $W$, respectively.

Let $G$  be a 2-arc-transitive subgroup of automorphisms of $\Gamma$ which can be lifted, and let $G^+\leq G$  be the corresponding index 2 normal subgroup of $G$ fixing $U$ and $W$ setwisely. Then $G^+\cong PSL(d, q)$.
Denote by
$\overline{G}$ and $\overline{G^+}$ the lifts of $G$ and $G^+$, respectively.
Then  $\overline{G^+}/N\cong  G^+$.

Since $N\cong \mathbb{Z}_p$ is an abelian normal subgroup of $\overline{G^+}$, it follows that  $C_{\overline{G^+}}(N)$ is a normal subgroup of $\overline{G^+}$ and $N\leq C_{\overline{G^+}}(N)$.
Since $C_{\overline{G^+}}(N)$ is a normal subgroup of $N_{\overline{G^+}}(N)$,
$C_{\overline{G^+}}(N)/N$ is a normal subgroup of $N_{\overline{G^+}}(N)/N=\overline{G^+}/N=G^+$ and as $G^+$ is a nonabelian simple group, it follows that
either $C_{\overline{G^+}}(N)/N=N_{\overline{G^+}}(N)/N$ or $C_{\overline{G^+}}(N)/N=1$. Consequently
either $C_{\overline{G^+}}(N)=N_{\overline{G^+}}(N)$ or $C_{\overline{G^+}}(N)=N$.

Suppose that $C_{\overline{G^+}}(N)=N$. Then
$$\overline{G^+}/N=\overline{G^+}/C_{\overline{G^+}}(N)\leq \Aut(N)\cong \mathbb{Z}_{p-1},$$
that is,
$$G^+=\overline{G^+}/N\leq  \mathbb{Z}_{p-1},$$
which is impossible. Thus
$$C_{\overline{G^+}}(N)=N_{\overline{G^+}}(N)=\overline{G^+}.$$

If $\Gamma$ is the graph $B'(\PG(d-1,q))$ where $d\geq 3$ and $q$ is a prime power, then
by a similar argument as in  \cite{DMMX-2021},   $\overline{\Gamma}$ is either $X(3,2)$ or the graph $\Gamma(d, q, r)$ where $r|q-1$, and so that (2) holds.

In the remaining, we assume that  $\Gamma$ is the graph $B(\PG(d-1,q))$ where $d\geq 3$ and $q$ is a prime power.


Suppose first that  $d = 3$. Then in the projective plane $\PG(2,q)$,  every pair of $1$-dimensional subspaces are contained
in exactly one $2$-dimensional subspace and every pair of $2$-dimensional subspaces intersect at exactly one $1$-dimensional subspace.
Hence  $\Gamma$ has girth 6.
Furthermore, any $1$-dimensional subspace $u\in  U$ is contained in  three $2$-dimensional subspaces, and say $w_i\in  W, i = 1, 2, 3$. Let $w$ be  a $2$-dimensional subspace  which does  not contain $u$, and consider the following 6-cycles: $C_{jk} = (u, w_j , w_j \cap w,w,w \cap w_k, w_k, u)$, where $j\neq k\in  \{1, 2, 3\}$.
By  \cite[Lemma 5.1 (iii)]{DMM-2008}, $G^+$ is transitive on the set of  6-cycles in $\Gamma$, it follows from
Lemma \ref{lem-cent-1} and the fact $C_{\overline{G^+}}(N)=\overline{G^+}$ that all 6-cycles have the same voltages. By calculation,
$\psi (C_{12}) = \psi (C_{13}) + \psi (C_{32})$, and so we have  $\psi (C_{jk}) = 0$. Hence all 6-cycles of $\Gamma$ have trivial voltages. However, since all the 6-cycles span the cycle space of $\Gamma$, it follows from Lemma \ref{lem-vanish} that  $\overline{\Gamma}$ is
disconnected, a contradiction.

Suppose next that $d\geq 4$.
Then in  the projective space $\PG(d-1,q)$,  every pair of $1$-dimensional subspaces are contained
in $\frac{q^{d-2}-1}{q-1}$  subspaces of dimension $(d-1)$ and every pair of  $(d-1)$-dimensional subspaces intersect at $\frac{q^{d-2}-1}{q-1}$ subspaces of dimension $1$. Thus  for any two
$(d-1)$-dimensional subspaces  $w_1, w_2\in  W$ we have $|w_1 \cap  w_2| =\frac{q^{d-2}-1}{q-1} \geq 3$, as $d\geq 4$.  Moreover, $\Gamma$ contains both 4-cycles and 6-cycles.

By \cite[Lemma 5.1 (iv)]{DMM-2008},  $G^+$ is  transitive on the
set of 4-cycles. Since
$C_{\overline{G^+}}(N)=\overline{G^+}$ and by Lemma \ref{lem-cent-1},  all 4-cycles have the same voltages. Since $|w_1 \cap  w_2| \geq 3$, it follows that $w_1 \cap  w_2$ has at least  three 1-dimensional subspaces, and say $u_i \in w_1 \cap  w_2$, where $i = 1, 2, 3$. For   the following 4-cycles: $C_{jk} = (w_1, u_j, w_2, u_k, w_1)$, where $j\neq k\in  \{1, 2, 3\}$, we have
$\psi(C_{12}) = \psi (C_{13}) + \psi (C_{32})$, and so $\psi (C_{jk}) = 0$. Hence the voltages of all 4-cycles
are trivial. Since any three $(d-1)$-dimensional subspaces intersect at some 1-dimensional subspaces, it follows  that
the voltages on all 6-cycles are trivial too. As 4-cycles and 6-cycles span the cycle space of $\Gamma$,
it follows from Lemma  \ref{lem-vanish} that the graph $\overline{\Gamma}$ is disconnected, a contradiction.

 This concludes
the proof.  \qed

\begin{lemma}\label{cyc-knn-nk2}
Let $\Gamma$ be a  $2$-arc-transitive bicirculant  as a regular cyclic cover of   $\K_{n,n}-n\K_2$ where  $n\geq 3$, with a non-trivial cyclic covering transformation group of order $d$.
Then  one of the following holds.

\begin{itemize}
\item[(1)] $n=4$,  $d=2$ and    $\Gamma=\K_{4}^{4}$;

\item[(2)] $n=4$, $d=3$ and  $\Gamma=GP(12,5)$;

\item[(3)] $n=4$,  $d=6$ and   $\Gamma=GP(24,5)$;

\item[(4)]  $n=5$,   $d=3$ and  $ \Gamma=X_2(3)$;

\item[(5)] $n=q+1\geq 6$ for some odd prime power $q$, and  $\Gamma=\K_{q+1}^{2d}$ for some $d\geq 2$   dividing $q-1$.
\end{itemize}

\end{lemma}
\proof  Since $\Gamma$ is a  $2$-arc-transitive graph  as a regular cyclic cover of  the graph $\K_{n,n}-n\K_2$ where  $n\geq 3$, it follows from   \cite[Theorem 1.1]{XD-2014}
that  $\Gamma$ is  one of the following graphs:

\begin{itemize}
\item[(1)] $n=4$ and  $\Gamma=\K_{4}^{4}$ for $d=2$; $X_1(3)$  for $d=3$;
or $X(6)$  for $d=6$;

\item[(2)]  $n=5$, $ \Gamma=X_2(3)$ for  $d=3$;

\item[(3)] $n=q+1\geq 6$ for some odd prime power $q=p^l$, and  $\Gamma=\K_{q+1}^{2d}$ for some $d\geq 2$   dividing $q-1$.
\end{itemize}

For the graph $X_2(3)$  in case (2) and the graph   $\K_{q+1}^{2d}$ in case (3), Lemmas \ref{lem-cov-k55} and \ref{lem-cov-comp-minus1} show that  both of these two graphs  are  $2$-arc-transitive bicirculants.

Now we argue  the graphs in case (1). Note  that $\K_{4}^{4}=X_1(4, 3)$ is the generalized Petersen graph $GP(8,3)$ (also the M\"obius-Kantor graph), see for instance  \cite[p.1364]{DMM-2008}. Consequently
$\K_{4}^{4}$ is a bicirculant.


The graph
$X_1(3)$ is the graph $CQ(1,3)(=GP(12,5))$ of \cite[p.721]{FW-2003} and $X(6)$ is the graph $CQ(1,6)(=GP(24,5))$ of \cite[p.721]{FW-2003}, see also for instance   \cite[Lemma 3.1]{XD-2014}. Thus these two graphs are bicirculants.

Moreover,  \cite{FW-2003} and \cite[Lemma 3.1]{XD-2014} also show that  $\K_{4}^{4}$,  $X_1(3)$ and  $X(6)$ are 2-arc-transitive.
 \qed

\bigskip

\bigskip

\section{Metacyclic Covers }

This subsection is devoted to finding the metacyclic covers of the basic $2$-arc-transitive bicirculants in case (2) of  Theorem \ref{bic-reduction-th2}.

\begin{lemma}\label{lem-index2-norm1}
Let $\Gamma$ be a  $2$-arc-transitive   cover of   $\Sigma$, with a non-trivial regular covering transformation group $N$.
Let  $M$  be  a characteristic subgroup of $N$. Then $M$ is a normal subgroup of $\Aut(\Gamma)$ and the following statements hold.

 \begin{itemize}
\item[(i)]  $\Gamma$ is a     cover of  $\Gamma_{M}$ with a regular covering transformation group $M$.

\item[(ii)]   $\Gamma_{M}$ is a $2$-arc-transitive cover of
 $\Sigma$, with a regular covering transformation group $N/M$.
\end{itemize}

\end{lemma}
\proof Let $A:=\Aut(\Gamma)$.  Since $N$ is a normal subgroup of $A$,
it follows that each $N$-orbit  is a block of $A$. Let  $\mathcal{B}=\{B_1,\ldots,B_t\}$ be the set of $N$-orbits.

Since  $N$ is regular on each orbit $B_i$,
it follows that $M$ is semiregular on $B_i$, and so $M$ is regular on each of its own orbit.
Thus each $B_i$ is the union of $f=|N/M|$ disjoint $M$-orbits, and say $B_i=C_{i1}\cup C_{i2}\cup \ldots \cup C_{if}$ where $C_{ij}$ is an $M$-orbit and $C_{ir}\cap C_{is}=\emptyset$.
Hence  $\mathcal{C}=\{C_{11}, C_{12},\ldots,C_{1f},C_{21},\ldots,C_{tf}\}$ is the set of $M$-orbits.

Since $M$  is  a  characteristic subgroup of $N$, it follows that $M$ is a normal subgroup of the group $A$. Thus  each element of $ \mathcal{C}$ is a block of $A$ and the set $\mathcal{C}$ is an $A$-invariant partition of $V(\Gamma)$.

We denote by  $\Gamma_{M}$  the quotient graph that induced by the set $\mathcal{C}$.
Suppose that $(C_{11},C_{21})$ is an arc of $\Gamma_{M}$ and  $c_{11}\in C_{11}$ is adjacent to a vertex $c_{21}\in C_{21}$. Assume that
there is another vertex $c_{12}\in C_{11}$ that is adjacent to $c_{2e}\in C_{2e}$ where $e\neq 1$. Then since each $C_{ij}$ is an $M$-orbit, there exists
$g\in M$ such that $c_{12}^g=c_{11}$. Hence $(c_{12},c_{2e})^g=(c_{11},c_{2e}^{g})$, and so
$c_{2e}^{g}\in C_{2e}$ and $c_{2e}^{g}\neq c_{21}$, this implies that  $\{c_{2e}^{g},c_{21}\}\subseteq \Gamma(c_{11})\cap B_2$,
and so $|\Gamma(c_{11})\cap B_2|\geq 2$. However, since $\Gamma$ is  a     cover of   $\Sigma\cong \Gamma_N$, we must have
$|\Gamma(c_{11})\cap B_2|=1$,  a contradiction.
Thus every vertex of $ C_{11}$ is  adjacent to non  vertex of $ C_{2e}$ where $e\neq 1$.
Since  $\Gamma$ is a cover of $\Gamma_{N}$ and $B_1,B_2$ are adjacent, it follows that  $[B_1\cup B_2]\cong |N|\K_2$,
and so  $[C_{11}\cup C_{21}]\cong |M|\K_2$.
Thus for any two sets $C_{ir}$ and $ C_{js}$,
the induced subgraph $[C_{ir}\cup C_{js}]$ is either a perfect matching  or edgeless.
Thus $\Gamma$ is a cover of $\Gamma_{M}$ with the regular covering transformation group $M$, and $\Gamma_{M}$  is a cover of $\Gamma_N\cong \Sigma$ with the regular covering transformation group $N/M$.

Let $C\in \mathcal{C}$ and $u$ be a vertex  in $ C$. Since $\Gamma$ is 2-arc-transitive, it follows  that  $A_u$ is 2-transitive on $\Gamma(u)$.
Then by the fact that $\Gamma$ is a cover of $\Gamma_{M}$ with the regular covering transformation group $M$, we can easily show that  $(A/M)_{C}$ is 2-transitive on $\Gamma_{M}(C)$, and so $\Gamma_{M}$ is $(A/M,2)$-arc-transitive.
\qed

\begin{lemma}\label{metacycl-lem-index2-3}
Let $\Gamma$ be a connected $(G,2)$-arc-transitive  bicirculant  of valency at least $3$ over   the    cyclic subgroup $H$ of $G$.
Let $H_0$ and $H_1$ be the two $H$-orbits on $V(\Gamma)$. Let $N$ be a  normal subgroup of $G$ maximal with respect to having at least $3$ orbits.
Suppose that each $N$-orbit intersects both $H_0$ and $H_1$ non-trivially.
Then $\Gamma$ and $N$ are in one of the following cases:

 \begin{itemize}

\item[(1)]    $ AT_D(4,6)$ with   $N\cong D_6$;

\item[(2)]  $ AT_D(5,6)$ with   $N\cong D_{6}$;

\item[(3)]  $ AT_Q(4,12)$ with   $N\cong Q_{12}$;

\item[(4)]  $ AT_Q(1+q,2d)$ with   $N\cong Q_{2d}$, where $d\mid q-1$ and $d\nmid \frac{1}{2}(q-1)$;

\item[(5)]  $ AT_D(1+q,2d)$ with   $N\cong D_{2d}$, where  $d\mid \frac{1}{2}(q-1)$ and $d\geq 2$.

\end{itemize}

\end{lemma}
\proof Let  $\mathcal{B}=\{B_1,\ldots,B_f\}$ be the set of $N$-orbits, so that $f\geqslant 3$. Since each $N$-orbit intersects both $H_0$ and $H_1$ non-trivially, it follows that $|B_i\cap H_0|=|B_i\cap H_1|$.
Further, the group $HN/N\cong H/(H\cap N)$ is regular on $\mathcal{B}$, and so
$H\cap N$ is semiregular on each $B\in \mathcal{B}$ with two orbits. Since $N$ is regular on each orbit, it follows that   $H\cap N$ is a cyclic index 2 normal subgroup of $N$.




Hence we can assume that $H\cap N=\langle a\rangle$, $N=\langle a,b\rangle$, where
$a^n=1,b^2=a^t,b^{-1}ab=a^r$,  $r^2\equiv 1\pmod{n}$ and $t(r-1)\equiv 0\pmod{n}$.
It follows that  $|b|=\frac{2n}{(t,n)}$.
Moreover, the element $b$ swops the intersection of $N$ with $H_0$ and the intersection of $N$ with $H_1$.
The derived subgroup of $N$ is $\langle a^{r-1}\rangle$, and this cyclic group   is a characteristic subgroup of $N$ and so it is a normal subgroup of $G:=\Aut(\Gamma)$.

Let $N':=\langle a^{r-1}\rangle$. Then $N'\leq H\cap N$ and so each orbit  of $N'$ is a  subset of $H_0$ or $H_1$, hence  $\Gamma_{N'}$ is a bicirculant over $H/H\cap N$.
Moreover,  by Lemma \ref{lem-index2-norm1}, for the  group $N'$, the projection $\Gamma \mapsto \Gamma_N$
can be factorized as
$$\Gamma \mapsto \Gamma_{N'} \mapsto \Gamma_N,$$
where
$\Gamma \mapsto \Gamma_{N'}$ means that $\Gamma$ is a $2$-arc-transitive regular cover of $\Gamma_{N'}$ with the covering transformation group  $N'$, and
$\Gamma_{N'} \mapsto \Gamma_N$ means that $\Gamma_{N'}$ is a $(G/N',2)$-arc-transitive regular cover of $\Gamma_{N}$ with the covering transformation group $N/N'$.

Let $K:=N/N'$. Then  $K$ is an abelian metacyclic group.
Since $N'=\langle a^{r-1}\rangle$, we can suppose  that   $K$ is a group with the  form $\langle e,b\rangle$, and  $\langle e,b\rangle/\langle e\rangle\cong \mathbb{Z}_2$.
If  $(|e|,2)=1$, then as $K$ is  abelian and $K/\langle e\rangle\cong \mathbb{Z}_2$, we have $K=\langle e\rangle\times \mathbb{Z}_{2}\cong \mathbb{Z}_{|e|}\times \mathbb{Z}_{2}$.
Further,  $K_1:=\langle e\rangle$ is a characteristic subgroup of $K$ and $K/K_1\cong \mathbb{Z}_2$.
Suppose that $|e|$ is an even integer.
Assume first that  $|e|=2^x$  is a 2-power. Then $K$ is a 2-group.
Let $K_2:=\langle e^{2^{x-1}},b\rangle$. Then $K_2\cong \mathbb{Z}_{2}\times \mathbb{Z}_2$ is a characteristic subgroup of $K$ and $K/K_2\cong \mathbb{Z}_{2^{x-1}}$.
Assume next that $|e|$ is an even integer but  not  a 2-power. Then $|K|$ has an  odd prime divisor  $p$. Set $K_3:=\langle e^p,b\rangle$.  Then $K_3$ is a characteristic subgroup of $K$,
and
$K/K_3\cong \mathbb{Z}_p$.

Thus $K$ always has  a characteristic subgroup $K_x$ such that $K/K_x$ is a cyclic group.
Since
$K$ is a normal subgroup of $G/N'$, it follows that
$K_x$ is a normal subgroup of $G/N'$.

Again, using Lemma \ref{lem-index2-norm1}, for the  group $K_x$, the projection $\Gamma_{N'} \mapsto \Gamma_N$
can be  factorized as
$$\Gamma_{N'} \mapsto \Gamma_{K_x} \mapsto \Gamma_N,$$
where
$\Gamma_{N'} \mapsto \Gamma_{K_x}$ means that $\Gamma_{N'} $ is a $2$-arc-transitive regular cover of $\Gamma_{K_x}$  with the covering transformation group  $K_x$, and $\Gamma_{K_x} \mapsto \Gamma_N$ means that $\Gamma_{K_x} $ is a $2$-arc-transitive regular cover of $\Gamma_{N} $ with the cyclic covering transformation group $K/K_x$.




By Theorem \ref{bic-reduction-th2}, we know that the quotient graph $\Gamma_{N}$ is one of the following three families of  graphs (circulants):

 \begin{itemize}

\item[(a)]  complete graph $\K_n$ with $n\geq 3$;

\item[(b)] complete bipartite graph $\K_{n,n}$, $n\geq 3$;

\item[(c)]  $\K_{n,n}-n\K_2$, and $n\geq 5$ is odd.

\end{itemize}



Suppose that  (a) holds, that is,  $\Gamma_N $ is  a complete graph $\K_{n}$ where  $n\geq 3$.
Then since  $\Gamma$ is  a  $2$-arc-transitive bicirculant of valency at least $3$  as a regular  cover of   $\K_{n}$, with a nonabelian metacyclic covering transformation group $N$, we have $n\geq 4$, and   it follows from   Theorem \ref{metacyclic-complete-th1} that
$\Gamma$ is isomorphic to one of the following graphs:

\begin{itemize}
\item[(1)]    $n=4$, $ AT_D(4,6)$ with   $N\cong D_6$;

\item[(2)] $n=4$, $ AT_Q(4,12)$ with   $N\cong Q_{12}$;

\item[(3)] $n=5$, $ AT_D(5,6)$ with   $N\cong D_{6}$;

\item[(4)] $n=1+q\geq 4$, $ AT_Q(1+q,2d)$ with   $N\cong Q_{2d}$, where $d\mid q-1$ and $d\nmid \frac{1}{2}(q-1)$;

\item[(5)] $n=1+q\geq 6$, $ AT_D(1+q,2d)$ with   $N\cong D_{2d}$, where  $d\mid \frac{1}{2}(q-1)$ and $d\geq 2$.
\end{itemize}



Suppose that case (b) occurs. Then $\Gamma_{N}$ is bipartite, and so $\Gamma_{K_x}$ is bipartite.  Since $\Gamma_{N'}$ is a 2-arc-transitive bicirculant cover of $\Gamma_{K_x}$,  it follows from Theorem \ref{bic-reduction-th2}   that $\Gamma_{K_x}$ is either a bicirculant or a circulant. If $\Gamma_{K_x}$ is  a circulant, then since
$\Gamma_{K_x}$ has even number of vertices, it follows that
$\Gamma_{K_x}$ is also  a bicirculant. Thus $\Gamma_{K_x}$ is always isomorphic to  a bicirculant.
However, by Lemma \ref{cyc-knn},  there are no $2$-arc-transitive bicirculants arising as regular cyclic covers of the graph  $\K_{n,n}$ where $n\geq 3$.
Thus $\Gamma_N$ is not isomorphic to $\K_{n,n}$ for any $n\geq 3$, and (b) does not occur.

Now  suppose that (c) holds, that is, $\Gamma_N=\K_{n,n}-n\K_2$ where  $n\geq 5$ is odd. Then as  $\Gamma_{K_x}$ is a $2$-arc-transitive cover of $\Gamma_{N}$ with cyclic covering transformation group $K/K_x$,  of prime order,
it follows from  Lemma \ref{cyc-knn-nk2} that  $n$, $|K/K_x|$ and the graph  $\Gamma_{K_x}$ are  one of the following cases:

\begin{itemize}
\item[(1)] $n=4$,  $d=|K/K_x|=2$ and    $\K_{4}^{4}$;

\item[(2)] $n=4$, $d=|K/K_x|=3$ and  $GP(12,5)$;

\item[(3)]  $n=5$,   $d=|K/K_x|=3$ and  $ X_2(3)$ ;

\item[(4)] $n=q+1\geq 6$ for some odd prime power $q=p^l$, and  $\K_{q+1}^{2d}$ for some $d=|K/K_x|\geq 2$   dividing $q-1$.
\end{itemize}

Since $n\geq 5$ is an odd integer, it follows that cases (1) and (2) do not occur.

If case (4) occurs, then   $\Gamma_{K_x}\cong \K_{q+1}^{2d}$ where $q+1\geq 6$ for some odd prime power $q=p^l$
However,  in this case  $n=d(q+1)$ is an even integer,
contradicts the fact that  $n\geq 5$ is odd.
Thus $\Gamma_{K_x}\ncong \K_{q+1}^{2d}$.

Finally assume that  case (3) occurs. Then  $\Gamma_{K_x}\cong X_2(3)$,  $n=15$  and
$\Gamma_N=\K_{5,5}-5\K_2$.
Further,   $\Gamma_{K_x}\cong X_2(3)$ implies that the covering transformation group  $K/K_x\cong \mathbb{Z}_3$.
Hence by the previous argument, we have  $K_x=\langle e^3,b\rangle$.
Recall that   $\Gamma_{N'}$ is a bicirculant over a subgroup $N'$ of $H$.
Let $(N')_0$ and $(N')_1$ be the two orbits of $N'$ on $V(\Gamma_N')$.
Then the element $b$ swops the intersection of $K_x$ with $(N')_0$ and the intersection of $K_x$ with $(N')_1$. Thus
each orbit of $K_x$ intersects both $(N')_0$ and $(N')_1$ non-trivially.
It follows that
$\Gamma_{K_x}$ is a circulant. However,  by \cite[Theorem 1.1]{ACMX-1996},
the graph $X_2(3)$ is a not a 2-arc-transitive circulant, a contradiction.
Thus $\Gamma_{K_x}\ncong X_2(3)$, and this leads to that  $\Gamma_N\neq \K_{n,n}-n\K_2$ for any odd integer $n\geq 5$. This completes the proof.
\qed


\bigskip

\bigskip

\section{Proof of Main Theorem }

\medskip
We are ready to prove the main theorem.

\medskip

\noindent {\bf Proof of Theorem \ref{th-quot-miniblock-1}.}
Let $\Gamma$ be a connected $2$-arc-transitive  bicirculant   over   the    cyclic subgroup $H$ of $G:=\Aut(\Gamma)$.
Then $\Gamma$ has at least 4 vertices.
If $\Gamma$ has valency 2, then $\Gamma$ is a cycle graph.  Since $\Gamma$ is a bicirculant with  at least 4 vertices, it follows that $\Gamma\cong C_{2n}$ for some $n\geq 2$.
In the remainder of this proof, we assume that  $\Gamma$ has valency at least 3.

Suppose that  $G$ is quasiprimitive on $V(\Gamma)$. Then it follows from
Proposition \ref{bicirculant-quasiprimitive-th1} that
$\Gamma$ is one of the following graphs:

\begin{center}
   $\K_{2n}$ with $n\geq 2$,  Petersen graph or its complement,

 $\H(2,4)$ or its complement,  folded $5$-cube or its complement.
\end{center}

Note that the following   graphs from the above list  are non-complete and have girth 3:  $\H(2,4)$ and its complement,  the complement of the folded $5$-cube and  the complement the Petersen graph. Thus these graphs  are not 2-arc-transitive. Hence  $\Gamma$ is one of the following three graphs:

\begin{center}
  $\K_{2n}$ with $n\geq 2$,   Petersen graph and    folded $5$-cube.
\end{center}

It remains to consider the case that   $G$ is not quasiprimitive on $V(\Gamma)$.
In this case $G$ has at least one intransitive normal subgroup, and this subgroup has at least 2 orbits on $V(\Gamma)$.

If  every non-trivial normal subgroup of $G$ has at most two orbits on $V(\Gamma)$ and there exists one which has exactly two orbits on $V(\Gamma)$, then $G$ is bi-quasiprimitive on $V(\Gamma)$ and  $\Gamma$ is a bipartite graph.
If the two bipartite halves are $H$-orbits, then by \cite[Proposition 5.1]{DGJ-2019},
  $\Gamma$ is one of the following graphs:
\begin{enumerate}[{\rm (a)}]
\item $\K_{n,n}$ where $n\geqslant 3$;
\item $ \K_{n,n}-n\K_2$ where $n\geqslant 4$;
\item  $B(H(11))$ and $B'(H(11))$;
\item $B(\PG(d-1,q))$ and $B'(\PG(d-1,q))$, where $d\geq 3$, $q$ is a prime power.
\end{enumerate}

If the two bipartite halves are not the  $H$-orbits, then by \cite[Proposition 5.2]{DGJ-2019},
$n$ is even and  $\Gamma$ is one of $\K_{n,n}$, $\K_{n,n}-n\K_2$, $B(\PG(d-1,q)$ or $B'(\PG(d-1,q))$ where $n=\frac{q^d-1}{q-1}$,   $d\geqslant 3$ is even and  $q$ is an odd prime power.

Now let $N$ be a  normal subgroup of $G$ maximal with respect to having at least $3$ orbits.
Let $H_0$ and $H_1$ be the two $H$-orbits on $V(\Gamma)$. Then by Theorem \ref{bic-reduction-th2},
$\Gamma$ is a cover of $\Gamma_N$,
$\Gamma_N$ is $(G/N,2)$-arc-transitive, and $G/N$ is faithful and either quasiprimitive or bi-quasiprimitive on $V(\Gamma_N)$.
 Moreover, one of the following holds.

\begin{enumerate}[$(1)$]
\item each $N$-orbit is a subset of  either $H_0$ or $H_1$,   $N$       is cyclic, and
 $\Gamma_N$ is one of the following graphs:

\begin{enumerate}[$(a)$]
\item  $\K_n$ where $n\geq 4$;

\item Petersen graph;

\item  folded $5$-cube;

\item  $\K_{n,n}$, $n\geq 3$;

\item $ \K_{n,n}-n\K_2$, $n\geq 4$;

\item  $B(H(11))$ and $B'(H(11))$;

\item $B(\PG(d-1,q))$ and $B'(\PG(d-1,q))$, where $d\geq 3$, $q$ is a prime power.

\end{enumerate}

\item each $N$-orbit intersects both $H_0$ and $H_1$ non-trivially, and  $N$       has a cyclic index $2$ normal subgroup, and $\Gamma_N$ is one of the following circulants:

\begin{enumerate}[$(a)$]
\item   $\K_n$ with $n\geq 4$;

\item  $\K_{n,n}$, $n\geq 3$;

\item  $\K_{n,n}-n\K_2$, and $n\geq 5$ is odd.

\end{enumerate}

\end{enumerate}


If (1)(a)  occurs, that is,   $\Gamma $ is a  cyclic cover of the complete graph  $\K_{n}$ where $n\geq 4$, then
by Lemma \ref{cyc-kn},  $\Gamma$ is either  $\K_{n,n}-n\K_2$ or
$ X_1(4,q)$ where $q\equiv 3\pmod{4}$ and $q=n-1$.

If (1)(b)  occurs, that is, $\Gamma $ is a  cyclic cover of the Petersen graph, then it follows from  Lemma \ref{cyccov-petersen} that  $\Gamma$
is either the  Desargues graph or the dodecahedron graph.

By Lemma \ref{cyccov-folded5cube},    $\Gamma $ is not a  cyclic cover of  the folded $5$-cube,
and so  (1)(c) does not occur.

Suppose that (1)(d) holds. Then Theorem \ref{bic-reduction-th2} further indicates that $\Gamma_N $ is a bicirculant over the cyclic group $H/N$ and the
two bipart halves are the orbits of $H/N$. Using Lemma \ref{cyc-knn},   $\Gamma $ is not a  cyclic cover of  such a  $\Gamma_N$, a contradiction.

By Lemma \ref{cyc-knn-nk2},
if  (1)(e)  occurs, that is, $\Gamma$ is a   cyclic cover of   $\K_{n,n}-n\K_2$ where  $n\geq 4$,
then  one of the following holds:

\begin{itemize}
\item[(i)]  $\Gamma=\K_{4}^{4}$, $GP(12,5)$,  $GP(24,5)$;

\item[(ii)]  $ \Gamma=X_2(3)$;

\item[(iii)]  $\Gamma=\K_{q+1}^{2d}$ for some $d\geq 2$   dividing $q-1$, and $n=q+1\geq 6$ for some odd prime power $q$.
\end{itemize}

Lemma \ref{cyc-H11} says that    $\Gamma $ is neither  a  cyclic cover of $B(H(11))$ nor $B'(H(11))$, and so (1)(f) does not occur.

Suppose that (1)(g) holds.
Then by Lemma \ref{cyc-pg} (1),   $\Gamma $ is not  a  cyclic cover of $B(\PG(d-1,q))$; and
by Lemma \ref{cyc-pg} (2),   if $\Gamma $ is   a  cyclic cover of  $B'(\PG(d-1,q))$, where $d\geq 3$, $q$ is a prime power, then $\Gamma$ is either $X(3,2)$ or the graph $\Gamma(d, q, r)$ where $r|q-1$.

Finally,  suppose that case (2) occurs. Then it follows from
Lemma \ref{metacycl-lem-index2-3} that
$\Gamma$ is isomorphic to one of the following graphs:

 \begin{itemize}

\item[(i)]    $ AT_D(4,6)$ with   $N\cong D_6$;

\item[(ii)]  $ AT_D(5,6)$ with   $N\cong D_{6}$;

\item[(iii)]  $ AT_Q(4,12)$ with   $N\cong Q_{12}$;

\item[(iv)]  $ AT_Q(1+q,2d)$ with   $N\cong Q_{2d}$, where $d\mid q-1$ and $d\nmid \frac{1}{2}(q-1)$;

\item[(v)]  $ AT_D(1+q,2d)$ with   $N\cong D_{2d}$, where  $d\mid \frac{1}{2}(q-1)$ and $d\geq 2$.

\end{itemize}

We conclude the proof. \qed

\end{document}